\documentclass{amsart}
\usepackage{mathrsfs}
\usepackage{amsmath}
\usepackage{epsfig}
\usepackage{amsfonts}
\usepackage{amssymb}
\usepackage{amsthm}
\usepackage{amscd}
\usepackage[all]{xy}
\usepackage{rotating}
\usepackage{lscape}
\usepackage{amsbsy}
\usepackage{verbatim}
\usepackage{moreverb}

\title{Fusion Rings of Loop Group Representations} 
\author{Christopher L. Douglas}
\thanks{The author was supported by a Miller Research Fellowship.}
\address{Department of Mathematics, University of California, Berkeley, CA 94720, USA}
\email{cdouglas@math.berkeley.edu}

\usepackage{amsmath}
\usepackage{amsfonts}
\usepackage{amssymb}
\usepackage{amsthm}
\usepackage{comment}
\usepackage{epsfig}
\usepackage{psfrag}
\usepackage{mathrsfs}
\usepackage{amscd}
\usepackage[all]{xy}
\usepackage{amsbsy}
\usepackage{multicol}
\usepackage{latexsym}

\newtheorem{thm}{Theorem}[section]
\newtheorem{prop}[thm]{Proposition}
\newtheorem{lemma}[thm]{Lemma}
\newtheorem{mlemma}[thm]{Main Lemma}
\newtheorem{cor}[thm]{Corollary}

\theoremstyle{definition}
\newtheorem{defn}[thm]{Definition}

\theoremstyle{remark}
\newtheorem{remark}[thm]{Remark}

\newcommand{\ingm}{\includegraphics[scale=.85]}

\newcommand{\nid}{\noindent}

\newcommand{\ra}{\rightarrow}

\newcommand{\nn}{\nonumber}

\newcommand{\ZZ}{\mathbb Z}

\newcommand{\lam}{\lambda}
\newcommand{\fg}{\mathfrak{g}}

\newcommand{\ft}{\mathfrak{t}}

\newcommand{\bs}{\backslash}
\newcommand{\Lv}{\mathrm{Lv}}
\newcommand{\st}{\, \vert \,}
\newcommand{\cdt}{\!\cdot\!}
\newcommand{\cdts}{\hspace{-1pt}\cdot\hspace{-.4pt}}

\begin{document}

\begin{abstract}

We compute the fusion rings of positive energy representations of the loop groups of the simple, simply connected Lie groups.

\end{abstract}

\vspace*{-20pt}
\maketitle

\vspace*{-20pt}
\section{Introduction}

The representation theory of loop groups has become a crossroads for such diverse areas as conformal field theory, operator algebras, low-dimensional topology, quantum algebra, and homotopy theory.  A crucial aspect of the theory is the existence of a product on the collection of representations of a fixed central extension of a loop group---this product can be variously described as fusion of conformal fields~\cite{verlinde}, as Connes fusion of bimodules over von Neumann algebras~\cite{connes,wassermann}, or as a Pontryagin product in twisted equivariant K-homology~\cite{fht1}.  Our purpose in this paper is to compute the fusion ring structure on the representations of the central extensions of the loop groups of simple, simply connected Lie groups.

The Grothendieck group of positive energy representations of the level $k$ central extension of the loop group of the simple, simply connected group $G$ is isomorphic to a free abelian group on the dominant weights of $G$ of level at most $k$.  The fusion ring $F_k[G]$ of these positive energy representations can be described as a quotient $R[G]/I_k$ of the representation ring of the group $G$ by the ``fusion ideal" $I_k$.  This ideal is generated by an infinite collection of representations associated to the dominant weights of $G$ of level greater than $k$.  In vivid contrast to this unwieldy presentation, Gepner~\cite{gepner} conjectured that the fusion ring is a global complete intersection, which is to say that the fusion ideal $I_k$ is generated by exactly $n$ representations, for a group $G$ of rank $n$.  Despite the existence of a detailed understanding of the complexification of the fusion ring~\cite{aharony}, only recently has there been progress toward a global perspective on the integral fusion ring.  By an elegant Gr\"obner basis analysis, Bouwknegt and Ridout~\cite{bouwrid} showed that for any group $G$ the fusion ideal $I_k$ can be generated by a collection of representations whose size grows exponentially with the level $k$---these representations are associated to a set of dominant weights that collectively encloses the level $k$ Weyl alcove within the Weyl chamber.  In the paper~\cite{douglas-sfi}, we used a spectral sequence for the twisted equivariant $K$-homology of the group $G$, imported along the Freed-Hopkins-Teleman~\cite{fht1,fht2} identification of this $K$-homology with the fusion ring, to prove that there exists a finite level-independent bound on the number of representations necessary to generate the fusion ideal.  In this sequel, we extend the methods developed there to permit a complete computation:

\begin{thm} \label{thm-fusion}
The fusion ring $F_k[G]$ of positive energy representations of the level $k$ central extension of the loop group of the simple, simply connected Lie group $G$ is given as follows:
\small

\begin{align}
F_k[SU(n+1)] &= R[SU(n+1)]/
\begin{aligned}[t]
\langle
[k \lam_1 + \lam_n],
[k \lam_1 + \lam_{n-1}],
\ldots,
[k \lam_1 + \lam_2],
[(k\!+\!1) \lam_1]
\rangle
\end{aligned} 
\displaybreak[0] \nn \\
F_k[Spin(2n+1)] &= R[Spin(2n+1)]/
\begin{aligned}[t]
\langle
&
[(k\!-\!1)\lam_1 + \lam_2],
[(k\!-\!1)\lam_1 + \lam_3],
\ldots,
[(k\!-\!1)\lam_1 + \lam_{n-1}],
\\
&
[(k\!-\!1)\lam_1 + 2 \lam_n],
[k \lam_1 + \lam_n],
[k \lam_1 + \lam_{n-1}],
\ldots,
[k \lam_1 + \lam_3],
\\
&
[k \lam_1 + \lam_2]+[k \lam_1],
[(k\!+\!1) \lam_1]
\rangle
\end{aligned} 
\displaybreak[0] \nn \\
F_k[Sp(n)] &= R[Sp(n)]/
\begin{aligned}[t]
\langle
[k \lam_1 + \lam_n],
[k \lam_1 + \lam_{n-1}],
\ldots,
[k \lam_1 + \lam_2],
[(k\!+\!1) \lam_1]
\rangle
\end{aligned} 
\displaybreak[0] \nn \\
F_k[Spin(2n)] &= R[Spin(2n)]/
\begin{aligned}[t]
\langle
&
[(k\!-\!1)\lam_1 + \lam_2],
[(k\!-\!1)\lam_1 + \lam_3],
\ldots,
[(k\!-\!1)\lam_1 + \lam_{n-2}],
\\
&
[(k\!-\!1)\lam_1 + \lam_{n-1} + \lam_n],
[k \lam_1 + \lam_n],
[k \lam_1 + \lam_{n-1}],
[k \lam_1 + \lam_{n-1} + \lam_n],
\\
&
[k \lam_1 + \lam_{n-2}],
[k \lam_1 + \lam_{n-3}],
\ldots,
[k \lam_1 + \lam_3],
[k \lam_1 + \lam_2] + [k \lam_1],
[(k\!+\!1) \lam_1]
\rangle
\end{aligned} \nn
\end{align}

\begin{align}
F_k[G_2] &= R[G_2]/
\begin{cases}
\langle
\left[3,\ell\!-\!1\right],
\left[1,\ell\right],
\left[0,\ell\!+\!1\right] + \left[0,\ell\right]
\rangle
& \text{if $k$ even} \vspace{4pt}
\\
\langle
\left[0,\langle k\!+\!1 \rangle\right],
\left[3,\langle k\!-\!1 \rangle\right] + \left[3,\langle k\!-\!3 \rangle\right],
\left[2,\langle k\!-\!1 \rangle\right]
\rangle
& \text{if $k$ odd}
\end{cases} 
\displaybreak[0] \nn \\
F_k[F_4] &= R[F_4]/
\begin{cases}
\begin{split}
\langle
&
\left[\ell,0,0,0\right]+\left[\ell\!+\!1,0,0,0\right],
\left[\ell\!-\!2,1,0,0\right]+\left[\ell,1,0,0\right],
\left[\ell\!-\!3,0,2,0\right]+\left[\ell,0,2,0\right],
\\
&
\left[\ell\!-\!2,0,1,1\right]+\left[\ell,0,1,1\right],
\left[\ell\!-\!1,0,1,0\right]+\left[\ell,0,1,0\right],
\left[\ell\!-\!3,1,0,2\right]+\left[\ell\!-\!1,1,0,2\right],
\\
&
\left[\ell\!-\!4,1,1,2\right]+\left[\ell\!-\!2,1,1,2\right],
\left[\ell\!-\!1,0,0,2\right]+\left[\ell,0,0,2\right],
\left[\ell\!-\!2,1,0,1\right]+\left[\ell\!-\!1,1,0,1\right],
\\
&
\left[\ell\!-\!2,0,1,2\right]+\left[\ell\!-\!1,0,1,2\right],
\left[\ell\!-\!3,2,0,0\right]+\left[\ell\!-\!2,2,0,0\right],
\left[\ell\!-\!3,1,0,3\right]+\left[\ell\!-\!2,1,0,3\right],
\\
&
\left[\ell\!-\!3,2,0,2\right]+\left[\ell\!-\!4,2,0,2\right],
\left[\ell\!-\!2,1,1,1\right]+\left[\ell\!-\!3,1,1,1\right],
\left[\ell\!-\!1,0,2,0\right]+\left[\ell\!-\!2,0,2,0\right],
\\
&
\left[\ell\!-\!1,0,2,1\right]+\left[\ell\!-\!3,0,2,1\right],
\left[\ell\!-\!1,1,0,0\right],
\left[\ell\!-\!2,1,0,2\right],
\left[\ell\!-\!1,0,1,1\right],
\left[\ell,0,0,1\right],
\\
&
\left[\ell\!-\!3,1,1,2\right],
\left[\ell\!-\!2,0,2,1\right]
\rangle
\text{      if $k$ even}
\end{split} \vspace{6pt}
\\ 
\begin{split}
\langle
&
\left[\langle k\!-\!3 \rangle,1,0,0\right]\!+\!\left[\langle k\!-\!1 \rangle,1,0,0\right],
\left[\langle k\!-\!5 \rangle,0,2,0\right]\!+\!\left[\langle k\!-\!1 \rangle,0,2,0\right],
\\
&
\left[\langle k\!-\!3 \rangle,0,1,1\right]\!+\!\left[\langle k\!-\!1 \rangle,0,1,1\right],
\left[\langle k\!-\!5 \rangle,1,0,2\right]\!+\!\left[\langle k\!-\!3 \rangle,1,0,2\right],
\\
&
\left[\langle k\!-\!7 \rangle,1,1,2\right]\!+\!\left[\langle k\!-\!5 \rangle,1,1,2\right],
\left[\langle k\!-\!5 \rangle,2,0,2\right]\!+\!\left[\langle k\!-\!9 \rangle,2,0,2\right],
\\
&
\left[\langle k\!-\!3 \rangle,1,1,1\right]\!+\!\left[\langle k\!-\!7 \rangle,1,1,1\right],
\left[\langle k\!+\!1 \rangle,0,0,1\right]\!+\!\left[\langle k\!-\!1 \rangle,0,0,1\right],
\\
&
\left[\langle k\!-\!1 \rangle,1,0,1\right]\!+\!\left[\langle k\!-\!5 \rangle,1,0,1\right],
\left[\langle k\!-\!1 \rangle,0,2,1\right]\!+\!\left[\langle k\!-\!7 \rangle,0,2,1\right],
\\
&
\left[\langle k\!+\!1 \rangle,0,1,0\right]\!+\!\left[\langle k\!-\!3 \rangle,0,1,0\right],
\left[\langle k\!+\!1 \rangle,1,0,0\right]\!+\!\left[\langle k\!-\!5 \rangle,1,0,0\right],
\\
&
\left[\langle k\!+\!3 \rangle,0,0,0\right]\!+\!\left[\langle k\!-\!1 \rangle,0,0,0\right],
\left[\langle k\!+\!1 \rangle,0,0,0\right],
\left[\langle k\!-\!1 \rangle,0,1,0\right],
\\
&
\left[\langle k\!-\!1 \rangle,0,0,2\right],
\left[\langle k\!-\!3 \rangle,1,0,1\right],
\left[\langle k\!-\!3 \rangle,0,1,2\right],
\left[\langle k\!-\!5 \rangle,2,0,0\right],
\\
&
\left[\langle k\!-\!5 \rangle,1,0,3\right],
\left[\langle k\!-\!3 \rangle,0,2,0\right],
\left[\langle k\!-\!7 \rangle,2,0,2\right],
\left[\langle k\!-\!5 \rangle,1,1,1\right]
\rangle
\text{      if $k$ odd}
\end{split}
\end{cases}
\displaybreak[0] \nn \\
F_k[E_6] &= R[E_6]/
\begin{aligned}[t]
\langle
&
[k\!-\!1,1,0,0,0,0],
[k\!-\!2,0,0,1,0,0],
[k\!-\!3,0,1,0,1,0],
[k\!-\!1,0,0,0,1,0],
\\
&
[k\!-\!2,0,1,0,0,1],
[k\!-\!3,0,1,0,1,1],
[k\!-\!1,0,1,0,0,0],
[k\!-\!2,0,0,1,0,1],
\\
&
[k\!-\!2,0,1,0,1,0],
[k,0,0,0,0,1],
[k\!-\!2,1,1,0,0,1],
[k\!-\!1,0,0,1,0,0],
\\
&
[k\!-\!1,0,1,0,0,1],
[k\!-\!1,1,0,0,1,0],
[k\!-\!1,1,0,0,1,1],
[k,0,0,0,1,0],
\\
&
[k,1,0,0,0,1],
[k,1,0,0,1,0],
[k,0,1,0,0,0],
[k\!+\!1,0,0,0,0,0],
\\
&
[k\!-\!1,1,0,0,0,1]+[1,0,0,0,0,0] \cdot \nu,
[k\!-\!3,0,1,1,0,1]-[1,1,0,0,0,0] \cdot \nu,
\\
&
[k\!-\!2,1,0,1,0,1]-[1,0,0,0,0,1] \cdot \nu,
[k\!-\!1,0,0,1,0,1]-[1,0,0,0,0,0] \cdot \nu,
\\
&
[k,1,0,0,0,0]+ \nu,
[k,0,0,1,0,0]- \nu
\rangle
\end{aligned}
\displaybreak[0] \nn \\
F_k[E_7] &= R[E_7]/
\begin{aligned}[t]
\langle
&
[1,0,0,0,0,0,k\!-\!1],
[0,0,1,0,0,0,k\!-\!2],
[0,0,0,1,0,0,k\!-\!3],
[0,1,0,0,1,0,k\!-\!4], 
\\
&
[0,0,0,0,1,0,k\!-\!2],
[0,1,0,0,0,1,k\!-\!3],
[0,1,0,0,1,1,k\!-\!5],
[0,1,0,0,0,0,k\!-\!1],
\\
&
[0,0,0,1,0,1,k\!-\!4],
[0,1,0,0,1,0,k\!-\!3],
[0,0,1,0,0,1,k\!-\!3],
[0,1,0,1,0,1,k\!-\!5],
\\
&
[0,0,0,1,0,0,k\!-\!2],
[0,0,0,0,0,1,k\!-\!1],
[1,1,0,0,0,1,k\!-\!3],
[0,0,1,1,0,1,k\!-\!5],
\\
&
[0,1,0,0,0,1,k\!-\!2],
[1,0,0,1,0,1,k\!-\!4],
[0,0,0,1,0,1,k\!-\!3],
[1,0,1,0,1,1,k\!-\!5],
\\
&
[0,0,1,0,0,0,k\!-\!1],
[0,0,1,0,1,1,k\!-\!4],
[1,0,0,0,1,0,k\!-\!2],
[1,0,0,0,1,1,k\!-\!3],
\\
&
[1,0,1,0,1,0,k\!-\!3],
[0,0,0,0,1,0,k\!-\!1],
[1,0,1,0,1,1,k\!-\!4],
[1,0,0,1,0,0,k\!-\!2],
\\
&
[0,0,1,0,1,0,k\!-\!2],
[1,1,0,0,0,0,k\!-\!1],
[1,0,0,1,1,0,k\!-\!3],
[1,1,0,0,0,1,k\!-\!2],
\\
&
[0,0,0,1,0,0,k\!-\!1],
[1,1,0,0,1,0,k\!-\!2],
[1,0,0,0,0,1,k\!-\!1],
[1,1,0,1,0,0,k\!-\!2],
\\
&
[1,0,0,0,1,0,k\!-\!1],
[0,1,1,0,0,0,k\!-\!1],
[0,1,0,0,0,0,k],
[1,1,1,0,0,0,k\!-\!1],
\\
&
[1,1,0,0,0,0,k],
[0,1,1,0,0,0,k],
[0,0,0,0,1,0,k],
[0,0,0,0,0,1,k],
[0,0,0,0,0,0,k\!+\!1],
\\
&
[1,0,0,0,0,1,k\!-\!2]+[1,0,0,0,0,0,0] \cdot \mu,
[0,1,1,0,0,1,k\!-\!4]-[1,1,0,0,0,0,0] \cdot \mu,
\\
&
[0,0,1,0,1,0,k\!-\!3]-[0,0,1,0,0,0,0] \cdot \mu,
[1,0,1,0,0,1,k\!-\!3]-[1,0,0,0,0,0,1] \cdot \mu,
\\
&
[0,0,1,0,0,1,k\!-\!2]-[1,0,0,0,0,0,0] \cdot \mu,
[1,0,0,1,0,1,k\!-\!3]+[1,0,0,0,0,0,1] \cdot \mu,
\\
&
[1,0,0,0,0,0,k]+\mu,
[1,0,0,1,0,0,k\!-\!1]+[0,0,0,0,0,0,1] \cdot \mu,
\\
&
[0,0,1,0,0,0,k]-\mu,
[0,0,0,1,0,0,k]+\mu
\rangle
\end{aligned}
\displaybreak[0] \nn
\end{align}
\normalsize

\nid The fusion ring of $E_8$ at even and odd levels is given in Table~1 
and Table~2 
in Section~\ref{sec-fe}.

\nid In the above equations for the classical groups, the expression $[\lam] \in R[G]$ refers to the irreducible representation of $G$ with highest weight $\lam$.  For the group $Spin(2n+1)$, the rank $n$ is assumed to be at least 3, and for the group $Spin(2n)$, the rank $n$ is assumed to be at least 4.  In the equations for the exceptional groups, the expression $[a_1, a_2, \ldots, a_n] \in R[G]$ refers to the irreducible representation of $G$ with highest weight $a_1 \lam_1 + a_2 \lam_2 + \cdots + a_n \lam_n$.  Throughout $\lam_i$ are the fundamental weights of $G$ in the Bourbaki ordering.  The abbreviations $\ell$ and $\langle m \rangle$ refer respectively to $\frac{k}{2}$ and $\frac{m}{2}$; the abbreviations $\nu$ and $\mu$ refer respectively to the representations $[k,0,0,0,0,0]$ and $[0,0,0,0,0,0,k]$.

\end{thm}

For the fusion rings of the exceptional groups, we have implicitly assumed that the level $k$ is large enough that the entries in the expressions $[a_1, a_2, \ldots, a_n]$ are all at least $-1$; the representation $[a_1, a_2, \ldots, a_n]$ is taken to be zero if some $a_i$ is exactly $-1$.  Thus for $F_4$, $k>5$; for $E_6$, $k>1$; for $E_7$, $k>3$; and for $E_8$, $k>19$.  For smaller levels, the fusion ideal is obtained by inducing the errant weights back into the Weyl chamber.  Thanks to the various available perspectives on representations of loop groups, Theorem~\ref{thm-fusion} could be rephrased as a computation of the twisted $G$-equivariant $K$-homology of the group $G$, or as a computation of the fusion ring of the Wess-Zumino-Witten rational conformal field theory over the group manifold $G$. 

The above descriptions of the fusion rings for type $A$ and type $C$ were known previously~\cite{bouwrid,boysalkumar,douglas-tkt, gepner,gepnerschwimmer} and a proof appears in~\cite{bouwrid}.  By exploiting Jacobi-Trudy determinantal identities, Bouwknegt and Ridout~\cite{bouwrid} also produced an explicit description of the fusion ideal in type $B$ whose complexity grows only linearly with the level.  In both type $B$ and type $D$, Boysal and Kumar~\cite{boysalkumar} wrote down a finite collection of representations which they conjecture generate the fusion ideal; they also describe an intriguing conjecture relating the fusion ideals for these types to a simple radical ideal. 

The paper is organized as follows.  In Section~\ref{sec-affcent}, we recall the computational methods developed in the preceding paper~\cite{douglas-sfi}, particularly those concerning the resolution of the fusion ring coming from the twisted Mayer-Vietoris spectral sequence.  We then prove our main lemma, that the fusion ideal is the induction image of the kernel of an affine induction map to the twisted representation module of the centralizer of any one affine vertex of the Weyl alcove; for brevity we refer to this kernel as the ``fusion kernel".  Section~\ref{sec-kernel} establishes a series of techniques for computing the fusion kernel.  These techniques rely on specialized bases, called affine Steinberg bases, for twisted representation modules; we describe these bases and discuss their relationship to classical Steinberg bases for representation rings.  In Section~\ref{sec-indclosed}, we compute the fusion kernel when there is an affine Steinberg basis that is closed with respect to induction to the chosen affine vertex of the alcove.  In Section~\ref{sec-central}, we compute the fusion kernel when the the affine vertex of the alcove is central.  In Section~\ref{sec-orthog}, we compute the fusion kernel when the edge of the Weyl alcove connecting the origin to the affine vertex is orthogonal to the affine wall of the alcove.  Section~\ref{sec-comp} harnesses the results of the preceding two sections to prove Theorem~\ref{thm-fusion}.

\subsubsection*{Acknowledgments}
We thank Eckhard Meinrenken and Andr\'e Henriques for illuminating correspondence, and Arun Ram and Lauren Williams for helpful pointers to the literature.

\section{Affine centralizers control the fusion ideal} \label{sec-affcent}

In this section we recall the resolution of the fusion ring $F_k[G]$ in terms of the twisted representation modules of centralizer subgroups of $G$, and note the abstract presentation of the fusion ring resulting from this resolution.  We then prove our main lemma, that the fusion ideal is the induction image of the kernel of affine induction to any one affine vertex of the Weyl alcove.  

Though we will briefly review the necessary notation and concepts, we are working throughout in the context established in our preceding paper~\cite{douglas-sfi}; the reader may want to refer there for background material and for more extensive exposition.  Fix a simple, simply connected Lie group $G$, with Lie algebra $\fg$; unless indicated otherwise, all of the following items refer to this Lie algebra.

\begin{align*}
D &= \text{affine Dynkin diagram} \\
\{\alpha_i\}_{i=0}^n &= \text{simple affine roots} \\
\{\lam_i\}_{i=1}^n &= \text{fundamental weights, generating the lattice } \Lambda_W \\
A &= \text{Weyl alcove} \\
w_i &= \text{Weyl reflection corresponding to } \alpha_i \\
w_{k \cdot \alpha_0} &= \text{reflection in the plane } {-}k {\textstyle \frac{\alpha_0}{2}} + \ker h_{\alpha_0}, \text{for } h_{\alpha_0} \,\text{the coroot of } \alpha_0 \\
W_S^k &= 
\begin{cases}
\langle \{w_i \st i \in S\} \rangle & \text{if } 0 \notin S \subset D, \\ 
\langle \{w_i \st i \in S \backslash 0\} \cup \{w_{k \cdot \alpha_0}\} \rangle & \text{if } 0 \in S \subset D 
\end{cases}
\\
F_S &= \text{face of } A \text{ fixed by } W_S^1 \\
Z_S &= \text{centralizer of } F_S \\
\rho_S &= \text{half sum of the positive roots of } Z_S \\
R_k[Z_S] &= \ZZ[\Lambda_W]^{W_S^k}, \,\text{the twisted representation module of } Z_S \\
\end{align*}

For a weight $\lam$ and a reflection group $U$ in the affine Weyl group, let $A_{\lam}^U$ denote the antisymmetrization of $\lam$ with respect to $U$.  The twisted Mayer-Vietoris spectral sequence for the twisted equivariant $K$-homology of the group $G$, reinterpreted in terms of twisted representation modules, has the following form:

\begin{prop} \label{prop-resolution}
There is a complex of $R[G]$-modules $\bigoplus_{S \subset D, \lvert S\rvert = n-p} R_k[Z_S]$ whose differential has components $d^{S,T}: R_k[Z_S] \ra R_k[Z_T]$, for $T=S \cup j_s$, given by twisted induction: 
\[ 
d^{S,T}\bigg(\left[\frac{A_{\mu + \rho_S}^{W_S^k}}{A_{\rho_S}^{W_S^0}}\right]\bigg) = (-1)^s \left[\frac{A_{\mu + \rho_T}^{W_T^k}}{A_{\rho_T}^{W_T^0}}\right] 
\]
This complex is acyclic except in degree zero, where it has homology the level $k$ fusion ring $F_k[G]$.
\end{prop}

\nid This proposition is discussed in Section~2.3 of~\cite{douglas-sfi} and is closely related to a resolution described by Meinrenken~\cite{meinrenken}---the original identification of the homology of the complex with the fusion ring is due to Freed, Hopkins, and Teleman~\cite{fht1}.  Investigating the cokernel of the first differential in this resolution produces a presentation of the fusion ring in terms of induction maps:

\begin{lemma} \label{lemma-quotient}
The level $k$ fusion ring $F_k[G]$ is isomorphic to the quotient of the representation ring $R[G]$ by the ideal 
$\langle d^{\widehat{01}, \widehat{0}} (\ker d^{\widehat{01}, \widehat{1}}), d^{\widehat{02}, \widehat{0}} (\ker d^{\widehat{02}, \widehat{2}}), \ldots, d^{\widehat{0n}, \widehat{0}} (\ker d^{\widehat{0n}, \widehat{n}}) \rangle$.
\end{lemma}

\nid Here $\widehat{S}$ denotes the complement of the collection $S$ in the affine Dynkin diagram $D$.  This lemma is established as part of the proof of Theorem~4.3 of~\cite{douglas-sfi}: first observe that the differential $d^{\widehat{0ij}, \widehat{ij}}: R_k[Z_{\widehat{0ij}}] \ra R_k[Z_{\widehat{ij}}]$ is surjective; the fusion ring is therefore the cokernel of the sum $\oplus_{i=1}^n (d^{\widehat{0i}, \widehat{0}} - d^{\widehat{0i}, \widehat{i}})$; second observe that the differential $d^{\widehat{0i},\widehat{i}}$ is also surjective; the lemma follows.  

The next lemma shows that the above presentation is redundant---indeed, the induction of the affine induction kernel for any one affine vertex of the Weyl alcove spans the fusion ideal:

\begin{mlemma} \label{lemma-onevert}
For any node $i$ of the Dynkin diagram of $G$, the level $k$ fusion ring is given by
\[F_k[G] = R[G] / \langle d^{\widehat{0i},\widehat{0}} (\ker d^{\widehat{0i},\widehat{i}}) \rangle\]
\end{mlemma}

Henceforth we will only be concerned with a single chosen affine vertex of the Weyl alcove.  We therefore introduce new and more specialized notation before proceeding to the proof of the lemma.

\begin{itemize}
\item[$\scriptstyle\bullet$] $o,v,e$ denote respectively the origin vertex of the alcove $A$, the affine vertex of the alcove associated to the simple root $\alpha_v$, and the edge of the alcove connecting those two vertices.
\item[$\scriptstyle\bullet$] $d^{e,o}: R[Z_e] \ra R[G]$ is the induction map $d^{\widehat{0v}, \widehat{0}}$.
\item[$\scriptstyle\bullet$] $d^{e,v}: R[Z_e] \ra R_k[Z_v]$ is the twisted induction map $d^{\widehat{0v}, \widehat{v}}$.
\item[$\scriptstyle\bullet$] $A_k$ is the level $k$ alcove, bounded by the fixed planes of the reflections $\{w_{k \cdot \alpha_0}, w_1, \ldots, w_n\}$.  
\item[$\scriptstyle\bullet$] $W^v$ is the reflection group $\langle w_{k \cdot \alpha_0}, w_1, \ldots, \widehat{w_v}, \ldots, w_n \rangle$.
\item[$\scriptstyle\bullet$] $W^e$ is the reflection group $\langle w_1, \ldots, \widehat{w_v}, \ldots, w_n \rangle$.
\item[$\scriptstyle\bullet$] $W_{-\rho}$ is the group generated by the reflections $\lam \mapsto w_i(\lam + \rho) - \rho$ for $i \in D \backslash \alpha_0$.
\item[$\scriptstyle\bullet$] $W_{-\rho_v}^v$ is the group generated by the reflections $\lam \mapsto w_i(\lam + \rho_v) - \rho_v$ for $v \neq i \in D \backslash \alpha_0$ and by the reflection $\lam \mapsto w_{k \cdot \alpha_0}(\lam + \rho_v) - \rho_v$.
\item[$\scriptstyle\bullet$] $W_{-\rho}^{k+h^{\vee}}$ is the reflection group generated by the group $W_{-\rho}$ and the groups $W_{-\rho_v}^v$ for all $v \in D \backslash \alpha_0$; the interior weights of the fundamental chamber for this group are the weights in the closed level $k$ alcove $A_k$.
\item[$\scriptstyle\bullet$] $C, C_v, C_e$ denote respectively the unique Weyl chambers for the reflection groups $W$, $W^v$, and $W^e$ containing $A_k$.
\item[$\scriptstyle\bullet$] $[\lam] \in R[G]$ and $[\lam]_e \in R[Z_e]$ denote the irreducible representations with highest weight $\lam$, for respectively $\lam \in C$ and $\lam \in C_e$.
\item[$\scriptstyle\bullet$] $[\lam]_v \in R_k[Z_v]$ denotes the twisted irreducible representation with highest weight $\lam$, for $\lam \in C_v$.
\end{itemize}

\begin{proof}
The ideal $DK := d^{e,o} (\ker d^{e,v}) = d^{\widehat{0i},\widehat{0}} (\ker d^{\widehat{0i},\widehat{i}})$ is certainly contained in the ideal $$\langle d^{\widehat{01}, \widehat{0}} (\ker d^{\widehat{01}, \widehat{1}}), d^{\widehat{02}, \widehat{0}} (\ker d^{\widehat{02}, \widehat{2}}), \ldots, d^{\widehat{0n}, \widehat{0}} (\ker d^{\widehat{0n}, \widehat{n}}) \rangle$$  By Lemma~\ref{lemma-quotient}, that latter ideal is the fusion ideal $I_k[G]$.  The fusion ideal is generated by the following collection of representations:
\[
\left\{[\lam] + (-1)^{a_{\lam}+1} [a_{\lam} \!\cdot\! \lam] \st \lam \in C \backslash (A_k \cup S_k)\right\} 
\cup 
\left\{[\lam] \st \lam \in C \cap S_k \right\}
\]
See for instance the papers~\cite{gepner} and~\cite{bouwrid}.  Here $S_k$ is the collection of weights on the $({-}\rho)$-shifted $(k+h^{\vee})$-scale Stiefel diagram walls, that is the weights with nontrivial isotropy in $W_{-\rho}^{k+h^{\vee}}$.  The transformation $a_{\lam} \in W_{-\rho}^{k+h^{\vee}}$ is the unique element such that $a_{\lam} \cdt \lam \in A_k$.

It suffices, therefore, to check that 
\begin{enumerate}
\item $[\lam] + (-1)^{a_{\lam}+1} [a_{\lam} \cdt \lam] \in DK$ for $\lam \in C \backslash (A_k \cup S_k)$, and 
\item $[\lam] \in DK$ for $\lam \in C \cap S_k$.  
\end{enumerate}
We establish both by induction on the length of the vector $\lam + \rho$.   For any weight $\lam$, let $v_{\lam} \in W_{-\rho_v}^v$ denote an element such that $v_{\lam} \cdt \lam \in C_v$, if one exists, and similarly let $c_{\lam} \in W_{-\rho}$ denote an element such that $c_{\lam} \cdt \lam \in C$, if one exists.  

For case (1), we have $\lam \in C \backslash (A_k \cup S_k)$.  Note that $[\lam]_e + (-1)^{v_{\lam} + 1} [v_{\lam} \cdt \lam]_e \in \ker d^{e,v}$.  It follows that
\begin{equation} \nn
[\lam] + (-1)^{c_{(v_{\lam} \cdts \lam)} + v_{\lam} + 1} [c_{(v_{\lam} \cdts \lam)} \cdt v_{\lam} \cdt \lam] \in DK
\end{equation}
Because $\lam$ is in the chamber $C$ but of level strictly larger than $k+1$, the vector $v_{\lam} \cdt \lam + \rho$ is strictly shorter than the vector $\lam + \rho$; the vector $c_{(v_{\lam} \cdts \lam)} \cdt v_{\lam} \cdt \lam + \rho$ is the same length as $v_{\lam} \cdt \lam + \rho$.  By induction
\begin{equation} \nn
[c_{(v_{\lam} \cdts \lam)} \cdt v_{\lam} \cdt \lam] + (-1)^{a_{(c_{(v_{\lam} \cdts \lam)} \cdts v_{\lam} \cdts \lam)} + 1}
[a_{(c_{(v_{\lam} \cdts \lam)} \cdts v_{\lam} \cdts \lam)} \cdt c_{(v_{\lam} \cdts \lam)} \cdt v_{\lam} \cdt \lam] \in DK
\end{equation}
Of course, the weight $a_{(c_{(v_{\lam} \cdts \lam)} \cdts v_{\lam} \cdts \lam)} \cdt c_{(v_{\lam} \cdts \lam)} \cdt v_{\lam} \cdt \lam$ just is $a_{\lam} \cdt \lam$.  Adding the above two displayed elements of the ideal $DK$ (after multiplying the second element by $(-1)^{c_{(v_{\lam} \cdts \lam)} + v_{\lam}}$) gives
\begin{equation} \nn
[\lam] + (-1)^{c_{(v_{\lam} \cdts \lam)} + v_{\lam} + a_{(c_{(v_{\lam} \cdts \lam)} \cdts v_{\lam} \cdts \lam)} + 1} [a_{\lam} \cdt \lam] \in DK
\end{equation}
Finally observe that the sum of orders $\lvert c_{(v_{\lam} \cdts \lam)} \rvert + \lvert v_{\lam} \rvert + \lvert a_{(c_{(v_{\lam} \cdts \lam)} \cdts v_{\lam} \cdts \lam)} \rvert$ and the order $\lvert a_{\lam} \rvert$ are congruent modulo $2$.  The whole induction begins with the situation $\lam \in A_k$; here $a_{\lam}$ is the identity and so evidently $[\lam] + (-1)^{a_{\lam} + 1} [a_{\lam} \cdt \lam] \in DK$.

Case (2), where $\lam \in C \cap S_k$, is analogous.  If $v_{\lam}$ does not exist, then $[\lam]_e \in \ker d^{e,v}$ and so $[\lam] \in DK$.  If $v_{\lam}$ exists but $c_{(v_{\lam} \cdts \lam)}$ does not, then again $[\lam]_e + (-1)^{v_{\lam} + 1} [v_{\lam} \cdt \lam]_e \in \ker d^{e,v}$ which implies $[\lam] \in DK$.  Those two cases begin the induction.  In general we find as before $[\lam] + (-1)^{c_{(v_{\lam} \cdts \lam)} + v_{\lam} + 1} [c_{(v_{\lam} \cdts \lam)} \cdt v_{\lam} \cdt \lam] \in DK$.  Again, provided $\lam$ is not level $k+1$ (in which case $v_{\lam}$ does not exist), the vector $c_{(v_{\lam} \cdts \lam)} \cdt v_{\lam} \cdt \lam + \rho$ is shorter than $\lam + \rho$.  By induction, $[c_{(v_{\lam} \cdts \lam)} \cdt v_{\lam} \cdt \lam] \in DK$, and so $[\lam] \in DK$ as desired.
\end{proof}

\section{Affine Steinberg bases and the fusion kernel} \label{sec-kernel} 

Lemma~\ref{lemma-onevert} shows that the fusion ideal is the induction image of the kernel of the affine induction map $d^{e,v}$.  In this section we establish techniques for computing that ``fusion kernel" $\ker d^{e,v}$.  The map $d^{e,v}: R[Z_e] \ra R_k[Z_v]$ is an $R[G]$-module map; we begin by discussing highest weight $R[G]$-module bases for $R[Z_e]$ that determine by restriction $R[G]$-module bases for $R_k[Z_v]$---these ``level $k$ affine Steinberg bases for $(R[Z_e],R_k[Z_v])$" will facilitate an analysis of the fusion kernel.  In Section~\ref{sec-indclosed}, we provide a complete description of the kernel in case there exists an affine Steinberg basis that is closed under the induction map $d^{e,v}$; such a basis exists for a vertex of the alcove of each of the classical groups and a vertex of the alcove of $G_2$.  In Section~\ref{sec-central}, we analyze the kernel when the vertex $v$ of the alcove is central, that is $Z_v = G$; this holds in particular for a vertex of the alcove of $E_6$ and a vertex of the alcove of $E_7$.  Finally in Section~\ref{sec-orthog}, we specify the kernel when the edge $e$ is orthogonal to the affine wall of the Weyl alcove; this is the case for a vertex of the alcove of $F_4$ and a vertex of the alcove of $E_8$.

Most of our discussion utilizes affine Steinberg bases for the representation modules associated to the centralizers $Z_e$ and $Z_v$.  More generally, the notion is as follows.

\begin{defn}
Let $S \subset T$ be two subsets of the affine Dynkin diagram of the simple, simply connected Lie group $G$; as before denote by $Z_S \subset Z_T$ the centralizers of the corresponding faces of the Weyl alcove.  Choose compatible Weyl chambers $C_T \subset C_S$ for the twisted representation modules $R_k[Z_T]$ and $R_k[Z_S]$ such that both chambers contain the level $k$ Weyl alcove $A_k$.  Two collections of weights $B_S \subset \Lambda_W$ and $B_T \subset \Lambda_W$, with $B_T = B_S \cap C_T$, form a \emph{level $k$ affine Steinberg basis} $(B_S, B_T)$ for the pair $(Z_S, Z_T)$ if the collection of twisted irreducible representations $\{[\lam]_S \in R_k[Z_S] \st \lam \in B_S\}$ is an $R[G]$-module basis for $R_k[Z_S]$, and if the collection of twisted irreducible representations $\{[\lam]_T \in R_k[Z_T] \st \lam \in B_T\}$ is an $R[G]$-module basis for $R_k[Z_T]$.
\end{defn} 
\nid Because, given the chosen Weyl chambers, the collection $B_T$ is determined by the collection $B_S$, we sometimes refer to $B_S$ by itself as the affine Steinberg basis for the pair $(Z_S, Z_T)$.

Though we expect that for any $S \subset T \subset D$, there exists a level $k$ affine Steinberg basis for the pair $(Z_S, Z_T)$, this remains an open problem.  Observe that Steinberg's construction of bases for representation rings of full rank subgroups of Lie groups~\cite{steinberg} depends only on working with reflection subgroups of the Weyl group.  If the twisted representation module $R_k[Z_T]$ is in fact untwisted, in the sense that it is isomorphic to $R[Z_T]$ as an $R[Z_T]$-module, then the corresponding Weyl group $W^k_T \subset W^{\text{aff}}$ is translation conjugate to a subgroup of the ordinary Weyl group.  Steinberg's arguments then apply mutatis mutandis to construct an affine Steinberg basis:

\begin{prop}
Suppose the representation module $R_k[Z_T]$ is untwisted, and therefore $R_k[Z_S]$ is also untwisted.  In this case, there exists a level $k$ affine Steinberg basis for the pair $(Z_S, Z_T)$.
\end{prop}

More concretely, a specific affine Steinberg basis can be described as follows.  Let $F^k_T$ denote the affine subspace of the dual torus $\ft^{\ast}$ fixed by the reflection group $W^k_T$; this plane is spanned by the face of the level $k$ alcove $A_k$ associated to the collection $T \subset D$.  If the representation module $R_k[Z_T]$ is untwisted, then the plane $F^k_T$ contains some element $\xi$ of the weight lattice of $G$---see~\cite{douglas-sfi} for a more detailed discussion of the circumstances under which representation modules are twisted or untwisted.  Choose Weyl chambers $\overline{C} \subset \overline{C}_T \subset \overline{C}_S$ for $R[G]$, $R[Z_T]$, and $R[Z_S]$ such that the shifted affine chambers $\overline{C}_T + \xi$ and $\overline{C}_S + \xi$ contain the level $k$ Weyl alcove $A_k$.  If $\bar{B}$ is the ordinary Steinberg basis for $R[Z_S]$ with respect to the pair of chambers $\overline{C} \subset \overline{C}_S$, then $\{\lam + \xi \st \lam \in \bar{B}\}$ is an affine Steinberg basis for the pair $(Z_S, Z_T)$.

\subsection{Induction closed Steinberg bases} \label{sec-indclosed} 

The fusion kernel has a particularly simple form when the collection of weights of the level $k$ affine Steinberg basis $B_e$ for the pair $(Z_e, Z_v)$ is closed under the induction map $d^{e,v}$.  Let $B_v$ denote the intersection $B_e \cap C_v$.

\begin{prop} \label{prop-indclosed}
Suppose $(B_e, B_v)$ is a level $k$ affine Steinberg basis for the centralizers $(Z_e, Z_v)$.  For any weight $\lam$, let $v_{\lam} \in W_{-\rho_v}^v$ be an affine Weyl element such that $v_{\lam} \cdt \lam \in C_v$, if one exists; if no such $v_{\lam}$ exists, we say the weight $\lam$ is $v$-singular, and denote by $S_v$ the collection of $v$-singular weights in $B_e$.  If for every weight $\lam \in B_e \backslash B_v$ either $\lam \in S_v$ or $v_{\lam} \cdt \lam \in B_v$, then the fusion kernel is
\begin{equation} \nn
\ker d^{e,v} = \left\langle \left\{[\lam]_e + (-1)^{v_{\lam} + 1} [v_{\lam} \cdt \lam]_e \st \lam \in B_e \backslash (B_v \cup S_v)\right\} \cup \big\{[\lam]_e \st \lam \in S_v\big\} \right\rangle
\end{equation}
\end{prop}

\begin{proof}

Evidently, the listed elements $\{[\lam]_e + (-1)^{v_{\lam} + 1} [v_{\lam} \cdt \lam]_e \st \lam \in B_e \backslash (B_v \cup S_v)\}$ and $\{[\lam]_e \st \lam \in S_v\}$ are in the kernel.  Given a representation $r \in R[Z_e]$ in the kernel of $d^{e,v}$, write
\begin{equation} \nn
r = \Bigg( \sum_{\lam \in S_v} r_{\lam} [\lam]_e \Bigg) + \Bigg( \sum_{\lam \in B_e \backslash S_v} r_{\lam} [\lam]_e \Bigg)
\end{equation}
Here for any weight $\lam$ the coefficient $r_{\lam}$ is in $R[G]$.  Because $d^{e,v} r = 0$, we have
\begin{equation} \nn
\sum_{\lam \in B_e \backslash S_v} (-1)^{v_{\lam}} r_{\lam} [v_{\lam} \cdt \lam]_v = 0
\end{equation}
For each $\mu \in B_v$, let $\mathcal{O}_{\mu}$ denote the collection of weights $\lam \in B_e$ such that $v_{\lam} \cdt \lam = \mu$.  For any such $\mu$, it follows that
\begin{equation} \nn
\sum_{\lam \in \mathcal{O}_\mu} (-1)^{v_{\lam}} r_{\lam} = 0
\end{equation}
Rewriting this coefficient equation as
\begin{equation} \nn
r_\mu = \sum_{\lam \in \mathcal{O}_\mu \backslash \mu} (-1)^{v_\lam + 1} r_\lam
\end{equation}
we find that
\begin{equation} \nn
\begin{split}
r 
&= \Bigg( \sum_{\lam \in S_v} r_{\lam} [\lam]_e \Bigg) + \Bigg( \sum_{\lam \in B_e \backslash S_v} r_{\lam} [\lam]_e \Bigg) \\
&= \Bigg( \sum_{\lam \in S_v} r_{\lam} [\lam]_e \Bigg) + \Bigg( \sum_{\lam \in B_e \backslash (S_v \cup B_v)} r_{\lam} \left( [\lam]_e + (-1)^{v_\lam + 1} [v_\lam \cdt \lam]_e \right) \Bigg)
\end{split}
\end{equation}
as desired.
\end{proof}

\subsection{Central Weyl alcove vertices} \label{sec-central}

Often the affine Steinberg basis $(B_e, B_v)$ for the centralizers $(Z_e, Z_v)$ has the property that $B_v$ is a small collection of weights, and so there is little hope that affine induction takes all the weights of $B_e$ into $B_v$, as considered in the last section.  However, if $B_v$ has exactly one weight, that is if the centralizer $Z_v$ is the whole group $G$, then again the kernel of induction to $Z_v$ has a particularly clean description.

\begin{prop} \label{prop-central}
Suppose the vertex $v$ of the Weyl alcove, corresponding to the simple root $\alpha_i$, is central in $G$, and $(B_e, B_v)$ is a level $k$ affine Steinberg basis for the pair $(Z_e, Z_v)$.  The collection $B_v$ is the single weight at that vertex of the level $k$ alcove, namely $k \frac{\lam_i}{h^{\vee}_i}$, where $h^{\vee}_i$ is the dual Coxeter label of the root $\alpha_i$; we abbreviate that weight by $\theta$.  In this case the fusion kernel is
\begin{equation} \nn
\ker d^{e,v} = \big\langle \big\{ [\lam]_e + (-1)^{v_\lam + 1} [w (v_\lam \cdt \lam - \theta)] \cdt [\theta]_e \st \lam \in B_e \backslash (B_v \cup S_v) \big\} \cup \big\{[\lam]_e \st \lam \in S_v\big\} \big\rangle
\end{equation}
Here $S_v$ is again the $v$-singular weights in $B_e$, and $w \in W$ is the Weyl element such that $w(C_v - \theta) = C$, that is the Weyl element aligning the Weyl chambers of the vertex and the origin of the alcove.
\end{prop} 

\begin{proof}
First check that for $\lam \in B_e \backslash (B_v \cup S_v)$, we have
$$d^{e,v}\left([\lam]_e + (-1)^{v_\lam + 1} [w (v_\lam \cdt \lam - \theta)] \cdt [\theta]_e\right)
= (-1)^{v_\lam} [v_\lam \cdt \lam]_v + (-1)^{v_\lam + 1} [w (v_\lam \cdt \lam - \theta)] \cdt [\theta]_v
= 0$$
Second given $r \in \ker d^{e,v}$, write
\begin{equation} \nn
r = \Bigg( \sum_{\lam \in S_v} r_{\lam} [\lam]_e \Bigg) + r_{\theta} [\theta]_e + \Bigg( \sum_{\lam \in B_e \backslash (B_v \cup S_v)} r_{\lam} [\lam]_e \Bigg)
\end{equation}
Inducing we have
\begin{equation} \nn
0 = r_{\theta} [\theta]_v + \sum_{\lam \in B_e \backslash (B_v \cup S_v)} (-1)^{v_\lam} r_{\lam} [v_\lam \cdt \lam]_v
= \Bigg( r_{\theta} + \sum_{\lam \in B_e \backslash (B_v \cup S_v)} (-1)^{v_\lam} r_{\lam} [w (v_\lam \cdt \lam - \theta)] \Bigg) [\theta]_v
\end{equation}
The coefficient of $[\theta]_v$ in this last expression is therefore $0$, and we conclude that
\begin{equation} \nn
r = \Bigg( \sum_{\lam \in S_v} r_{\lam} [\lam]_e \Bigg) + \Bigg( \sum_{\lam \in B_e \backslash (B_v \cup S_v)} r_\lam \left([\lam]_e + (-1)^{v_\lam + 1} [w (v_\lam \cdt \lam - \theta)] \cdt [\theta]_e\right) \Bigg)
\end{equation}
as desired.
\end{proof}

\subsection{Kernels of orthogonal induction} \label{sec-orthog}

For every simple, simply connected Lie group except $SU(n)$, $n > 2$, there is exactly one edge of the Weyl alcove perpendicular to the affine wall of the alcove.  The kernel of induction from that edge to the corresponding affine vertex of the alcove has a concrete description in terms of a weight basis for the representation ring $R[Z_e]$ of the centralizer of that edge.  We begin, though, with an intermediate description of this fusion kernel in terms of a generating set for a quotient of $R[Z_e]$ by irreducible representations associated to weights nonsingular for affine induction.

\begin{lemma} \label{lemma-orthsing}
Assume that the edge $e$ of the Weyl alcove is perpendicular to the affine wall, and has affine vertex $v$.  Let $B_e$ be a collection of weights in the chamber $C_e$ such that $\{[\lam]_e \st \lam \in B_e\}$ is an $R[G]$-module basis for $R[Z_e]$.  Let $S$ denote the collection of all weights in $C_e$ of level $k+1$, and let $R[S] \subset R[Z_e]$ denote the $\ZZ$-module span of $\{[\lam]_e \st \lam \in S\}$.  The module $R[S]$ is also the quotient $q$ of $R[Z_e]$ by the $\ZZ$-submodule generated by $\{[\lam]_e \st \lam \in C_e \backslash S\}$.  For $\epsilon \in R[G]$ and $\eta \in R[S]$, the operation $\epsilon \cdot \eta := q(\epsilon \eta) \in R[S]$ does not define an $R[G]$-action, but it does define an $R[G]$-action on the quotient $R[S]/(2)$.  Suppose $B_s \subset S$ is a collection of weights such that the quotients $\{[\lam]_e/(2) \st \lam \in B_s\}$ generate $R[S]/(2)$ as an $R[G]$-module.  Then the fusion kernel is
\begin{equation} \nn
\ker d^{e,v} = \big\langle \big\{[\lam]_e \st \lam \in B_s\big\} \cup \big\{[\lam]_e \st \lam \in B_e \cap S\big\} \cup \big\{[\lam]_e + [w_{(k+1) \cdot \alpha_0} \lam]_e \st \lam \in B_e \backslash (B_e \cap S)\big\} \big\rangle
\end{equation}
Here $w_{(k+1) \cdot \alpha_0}$ is reflection in the affine wall of the level $k+1$ alcove.
\end{lemma}

\begin{proof}
The listed elements are certainly in the kernel of $d^{e,v}$.  For a weight $\lam \in C_v$, the induction is $d^{e,v}([\lam]_e) = [\lam]_v$; for singular weights, the induction is zero; for a weight $\lam \in C_e \backslash (C_v \cup S)$, the induction is $d^{e,v}([\lam]_e) = - [w_{(k+1) \cdot \alpha_0} \lam]_v$.  The collection 
$$\big\{[\lam]_e \st \lam \in S\big\} \cup \big\{[\lam]_e + [w_{(k+1) \cdot \alpha_0} \lam]_e \st \lam \in C_e \backslash (C_v \cup S)\big\}$$ 
is therefore a $\ZZ$-basis for the kernel of $d^{e,v}$.  It suffices to check that each of these basis elements is in the $R[G]$-span of the collection 
$$\big\{[\lam]_e \st \lam \in B_s\big\} \cup \big\{[\lam]_e \st \lam \in S_v\big\} \cup \big\{[\lam]_e + [w_{(k+1) \cdot \alpha_0} \lam]_e \st \lam \in B_e \backslash S_v\big\}$$  
Here we have written $S_v$ for the $v$-singular basis weights $B_e \cap S$.

For $\mu \in C_e \backslash (C_v \cup S)$, write
\begin{equation} \nn
[\mu]_e = \Bigg( \sum_{\lam \in S_v} r_{\lam} [\lam]_e \Bigg) + \Bigg( \sum_{\lam \in B_e \backslash S_v} r_{\lam} [\lam]_e \Bigg)
\end{equation}
Next observe that
\begin{equation} \nn
\begin{split}
[\mu]_e + [w_{(k+1) \cdot \alpha_0} \mu]_e 
&= \Bigg( \sum_{\lam \in S_v} r_{\lam} [\lam]_e \Bigg) + \Bigg( \sum_{\lam \in B_e \backslash S_v} r_{\lam} [\lam]_e \Bigg) + \Bigg( \sum_{\lam \in S_v} r_{\lam} [\lam]_e \Bigg) + \Bigg( \sum_{\lam \in B_e \backslash S_v} r_{\lam} [w_{(k+1) \cdot \alpha_0} \lam]_e \Bigg) \\
&= \Bigg( \sum_{\lam \in S_v} 2 r_{\lam} [\lam]_e \Bigg) + \Bigg( \sum_{\lam \in B_e \backslash S_v} r_{\lam} \left([\lam]_e + [w_{(k+1) \cdot \alpha_0} \lam]_e \right) \Bigg)
\end{split}
\end{equation}
as desired.

For $\mu \in S$, by assumption there exist representations $r_\lam \in R[G]$ such that $[\mu]_e - \sum_{\lam \in B_s} r_\lam [\lam]_e = 0$ in $R[S]/(2)$.  Thus, there exist integers $t_\lam \in \ZZ$ such that 
$$[\mu]_e - \sum_{\lam \in B_s} r_\lam [\lam]_e - \sum_{\lam \in S} 2 t_\lam [\lam]_e = 0$$ 
in $R[S]$.  Because the representations $r_\lam \in R[G]$ are $w_{0 \cdot \alpha_0}$-symmetric, there furthermore exist integers $s_\lam \in \ZZ$ such that in $R[Z_e]$ we have
\begin{equation} \nn
[\mu]_e - \sum_{\lam \in B_s} r_\lam [\lam]_e - \sum_{\lam \in C_e \backslash (C_v \cup S)} s_\lam \left([\lam]_e + [w_{(k+1) \cdot \alpha_0} \lam]_e\right) - \sum_{\lam \in S} 2 t_\lam [\lam]_e = 0
\end{equation}
The first sum is in the required $R[G]$-span.  By the previous paragraph, the second sum is also in that span.  For the third sum, for each $\lam \in S$ write 
$$[\lam]_e = \sum_{\nu \in S_v} r_\nu [\nu]_e + \sum_{\nu \in B_e \backslash S_v} r_\nu [\nu]_e$$ with $r_\nu \in R[G]$, from which we have 
$$2 [\lam]_e = [\lam]_e + [w_{(k+1) \cdot \alpha_0} \lam]_e = \sum_{\nu \in S_v} 2 r_\nu [\nu]_e + \sum_{\nu \in B_e \backslash S_v} r_\nu \left([\nu]_e + [w_{(k+1) \cdot \alpha_0} \nu]_e\right)$$
as needed.
\end{proof}

Next we observe that the basis $B_s$ for the singular weight module $R[S]/(2)$, as in the above lemma, is determined by the basis $B_e$ for $R[Z_e]$.  To simplify the statement, we assume the group $G$ is not $SU(2)$ or $Sp(n)$---this restriction ensures that the fundamental weight on the edge $e$ has level $2$.

\begin{lemma} \label{lemma-bs}
Suppose the edge $e$ of the Weyl alcove is perpendicular to the affine wall, and has affine vertex $v$ corresponding to the simple root $\alpha_i$.  Assume the group $G$ is not $SU(2)$ or $Sp(n)$.  Let $B_e$ be a collection of weights in the chamber $C_e$ such that $\{[\lam]_e \st \lam \in B_e\}$ is an $R[G]$-module basis for $R[Z_e]$.  Then for $B_s$ as follows, the collection $\{[\lam]_e/(2) \st \lam \in B_s\}$ generates, over $R[G]$, the modulo 2 singular weight module $R[S]/(2)$:
\begin{equation} \nn
B_s = \Big\{ \lam + \frac{k+1-\Lv(\lam)}{2} \lam_i \st \lam \in B_e, \Lv(\lam) = k+1 \: (\mathrm{mod}\; 2)\Big\}
\end{equation}
Here $\Lv(\lam)$ is the level of the weight $\lam$.
\end{lemma}

\begin{proof}
Consider the series of maps
\begin{equation} \nn
R[Z_e] \ra R[Z_e]_{S\text{-cong}} \ra R[S] \ra R[S]/(2)
\end{equation}
Here $R[Z_e]_{S\text{-cong}}$ is the quotient of $R[Z_e]$ by the $\ZZ$-submodule generated by $\{[\lam]_e \st \, \Lv(\lam) \neq \linebreak k+1 \: (\mathrm{mod}\; 2)\}$, that is the quotient by all representations whose level is not congruent to the singular level.  The map $R[Z_e]_{S\text{-cong}} \ra R[S]$ takes a representation $[\lam]_e$, for $\Lv(\lam) = k+1 \: (\mathrm{mod}\; 2)$, to the representation $[\lam + \frac{k+1-\Lv(\lam)}{2} \lam_i]_e$.  Though neither $R[Z_e]_{S\text{-cong}}$ nor $R[S]$ is an $R[G]$-module, the composite $R[Z_e] \ra R[S]/(2)$ is a surjective $R[G]$-module map.  The representations $\{[\lam]_e \st \lam \in B_e\}$ form an $R[G]$-basis for $R[Z_e]$; this collection of representations maps to the collection 
$$\Big\{[\lam + \frac{k+1-\Lv(\lam)}{2} \lam_i]_e \in R[S]/(2) \st \lam \in B_e, \Lv(\lam) = k+1 \: (\mathrm{mod}\; 2)\Big\}$$
which therefore generates the quotient module $R[S]/(2)$.
\end{proof} 

\nid The cases $SU(2)$ and $Sp(n)$ present no difficulty---one need only replace the factor $\lam_i$ by $2 \lam_i$ throughout.  

Altogether, then, the basis $B_e$ determines the fusion kernel:

\begin{cor} \label{cor-orthog}
If the edge $e$ of the Weyl alcove, with affine vertex $v$ corresponding to the simple root $\alpha_i$, is perpendicular to the affine wall, and the group $G$ is not $SU(2)$ or $Sp(n)$, then the fusion kernel is given by
\begin{equation} \nn
\begin{split}
\ker d^{e,v} = \big\langle & \big\{ [\lam + \frac{k+1-\Lv(\lam)}{2} \lam_i]_e \st \lam \in B_e, \Lv(\lam) = k+1 \: (\mathrm{mod}\; 2) \big\} \\
& \cup \big\{[\lam]_e \st \lam \in B_e \cap S\big\} \\
& \cup \big\{[\lam]_e + [w_{(k+1) \cdot \alpha_0} \lam]_e \st \lam \in B_e \backslash (B_e \cap S)\big\} \big\rangle
\end{split}
\end{equation}
As before $B_e$ is a collection of weights in the chamber $C_e$ whose corresponding irreducible representations form an $R[G]$-basis for $R[Z_e]$, and $S$ is the set of level $k+1$ weights of $C_e$.
\end{cor}

\section{Computation of the fusion ideals} \label{sec-comp}

Section~\ref{sec-affcent} showed that the fusion ideal is the image $d^{e,o}(\ker d^{e,v})$.  Here, $d^{e,o}: R[Z_e] \ra R[G]$ is induction from the centralizer of an edge of the Weyl alcove, and $d^{e,v}: R[Z_e] \ra R_k[Z_v]$ is twisted induction from that edge to the twisted representation module of the corresponding vertex of the alcove.  Section~\ref{sec-kernel} computed the kernel $\ker d^{e,v}$ in terms of appropriate affine Steinberg bases for the centralizers $Z_e$.  In this section, we describe the necessary Steinberg bases and compute the induction images of the fusion kernels, for all simple, simply connected groups.

We begin by recalling a convenient form of Steinberg's original construction of bases for representation modules:
\begin{prop}[\cite{steinberg}]
Suppose $E \subset F$ are two full rank subgroups of the simple, simply connected group $G$.  Let $R_E \subset R_F \subset R$ be the root systems for the three groups, with Weyl groups $W_E$, $W_F$, and $W$.  Denote by $R[E] = \ZZ[\Lambda_W]^{W_E}$, $R[F] = \ZZ[\Lambda_W]^{W_F}$, and $R[G] = \ZZ[\Lambda_W]^{W}$ the three representation rings.  Choose a set of positive roots $R^+$, and let $D(R^+)$ be the simple roots of $R^+$.  Let $R_E^+ = R^+ \cap R_E$ and $R_F^+ = R^+ \cap R_F$ be the compatible choices of positive roots for the subgroups.  Denote by $C_E$, $C_F$, and $C$ the corresponding Weyl chambers.  Let
\[W^{E_+} := \{w \in W \st w(R_E^+) > 0\}\]
be the subgroup of the Weyl group of $G$ keeping the positive roots of $E$ positive, that is in $R^+.$  For $w \in W^{E_+}$, set
\[\lam_w = \sum_{\{\alpha \in D(R^+) \st w^{-1} \alpha < 0\}} \lam_\alpha\]
The irreducible representations corresponding to the collection of weights
\[B_E := \{w^{-1} \lam_w \st w \in W^{E_+} \}\]
form an $R[G]$-module basis for the representation ring $R[E]$.  Furthermore, the irreducible representations corresponding to the weights $B_F := B_E \cap C_F$ form an $R[G]$-module basis for the representation ring $R[F]$.
\end{prop} 
We refer to the collection $W^{E_+}$ as the ``$E$-positive Weyl group", and to its elements as ``Steinberg Weyl elements".  The corresponding weights $\Lambda^{E_+} := \{\lam_w \st w \in W^{E_+}\}$ are called the ``$E$-positive weights".  We say that the collection $B_E$ or the pair $(B_E, B_F)$ is a ``Steinberg basis" for the pair of groups $(E,F)$, with respect to the given choice of positive roots, and the weights of $B_E$ are ``Steinberg weights".  We will also utilize the following shorthand: for any full rank subgroup $H \subset G$ and a collection of simple roots $\Delta$ for $H$, we let $H[\Delta]$ denote that subgroup equipped with the choice of positive roots determined by $\Delta$.

We now overview the construction of affine Steinberg bases for certain centralizers in the simple, simply connected groups.  Fix the group $G$; let $D$ be the affine Dynkin diagram, with nodes labelled in the Bourbaki order, and let $\bar{D}$ be the finite Dynkin diagram.  Choose an affine vertex $v$ of the Weyl alcove, corresponding to a simple root $\alpha_i \in \bar{D}$.  Let $e$ denote the edge connecting this vertex to the origin.  We will encounter exactly two cases: (1) the roots $D \bs \alpha_i$ generate the root system of $G$; (2) the affine root $\alpha_0$ is orthogonal to the root system $\bar{D} \bs \alpha_i$.  In each case, we begin by calculating a Steinberg basis $(B_e, B_v)$ for the pair $(Z_e, Z_v)$ with respect to a particular choice of positive roots, distinguished according to the cases as follows:
\begin{itemize}
\item[(1)] $Z_e[\bar{D} \bs \alpha_i] \subset Z_v[\bar{D}] \subset G[\bar{D}]$
\item[(2)] $Z_e[\bar{D} \bs \alpha_i] \subset Z_v[(\bar{D} \cup -\alpha_0) \bs \alpha_i] \subset G[\bar{D}]$
\end{itemize}
Next, respectively in each of the two cases, we identify a transformation $r$ that reflects these positive root systems for $Z_e$ and $Z_v$ into the following systems:
\begin{itemize}
\item[(1)] $Z_e[\bar{D} \bs \alpha_i] \subset Z_v[D \bs \alpha_i]$
\item[(2)] $Z_e[\bar{D} \bs \alpha_i] \subset Z_v[D \bs \alpha_i]$
\end{itemize}
In case (2) this reflection is always $w_0 := w_{0 \cdot \alpha_0}$.  

Now provided the representation module $R_k[Z_v]$ is untwisted, the affine transformation
\[(AB_e, AB_v) := \left(r(B_e) + k \frac{\lam_i}{h^{\vee}_i}, r(B_v) + k \frac{\lam_i}{h^{\vee}_i}\right)\]
of the Steinberg basis $(B_e, B_v)$ is a level $k$ affine Steinberg basis for the pair of centralizers $(Z_e, Z_v)$.  That module is untwisted exactly when the dual Coxeter label $h^{\vee}_i$ divides the level---see Proposition 4.2 in~\cite{douglas-sfi}.  This will be true in all our cases except for the groups $G_2$, $F_4$, and $E_8$, at odd levels.  For $G_2$ at odd level we use a different affine transformation to produce an affine Steinberg basis; for $F_4$ and $E_8$ we observe that the collection $r(B_e) + k \frac{\lam_i}{h^{\vee}_i}$ remains a Steinberg basis for the centralizer $Z_e$, and this is sufficient to proceed with the computation.  For every group, we will find the Steinberg or affine Steinberg basis $AB_e$ is of the type considered in at least one of the Sections~\ref{sec-indclosed},~\ref{sec-central}, or~\ref{sec-orthog}.  The techniques considered there will serve to compute the fusion kernel, from which we finally determine the fusion ideal.

\subsection{The classical groups and $G_2$} \label{sec-classical}

\subsubsection*{Type A} \label{sec-a}

Let $v$ be the vertex of the alcove corresponding to the simple root $\alpha_1$; this root is a terminal root in the nonaffine Dynkin diagram.  Choose $\bar{D} \bs \alpha_1$ as the positive simple roots of the centralizer $Z_e$.  The $Z_e$-positive Weyl group is
\begin{equation} \nn
W^{{Z_e}_+} = \{1, w_1, w_2 w_1, w_3 w_2 w_1, \ldots, w_n w_{n-1} \cdots w_1\}
\end{equation}
The corresponding collection of $Z_e$-positive weights is
\begin{equation} \nn
\Lambda^{{Z_e}_+} = \{0, \lam_1, \lam_2, \lam_3, \ldots, \lam_n\}
\end{equation}
and the resulting Steinberg basis is
\begin{equation} \nn
B_e = \{0, w_1 \lam_1, w_1 w_2 \lam_2, w_1 w_2 w_3 \lam_3, \ldots, w_1 w_2 \cdots w_n \lam_n\}
\end{equation}
The centralizer $Z_v$ is the group $G$, and therefore $B_v = \{0\}$.

Let $r$ denote the automorphism of $\ft^{\ast}$ determined by the nontrivial order $2$ automorphism of the affine Dynkin diagram that exchanges $\alpha_1$ and $\alpha_0$.  This automorphism preserves the positive root system of $Z_e$ while transforming the positive simple roots $\bar{D}$ of $Z_v$ into the roots $D \bs \alpha_1$.  The pair
\[(AB_e, AB_v) = (r(B_e) + k \lam_1, r(B_v) + k \lam_1)\]
is a level $k$ affine Steinberg basis for $(Z_e, Z_v)$.  

The single weight $AB_v = \{k \lam_1\}$ is level $k$, and the remaining weights $AB_e \bs AB_v$ are all level $\mbox{k+1}$, therefore singular for induction to $R_k[Z_v]$.  In particular, the affine Steinberg basis is induction closed in the sense of Proposition~\ref{prop-indclosed}.  The fusion kernel is therefore $\{[\lam]_e \st \lam \in AB_e \bs AB_v\}$, and by Lemma~\ref{lemma-onevert}, the fusion ideal is $\{[\lam]_v \st \lam \in AB_e \bs AB_v\}$; the specific representations in the fusion ideal are listed in Theorem~\ref{thm-fusion}.

\subsubsection*{Type B} \label{sec-b}

Let $v$ be the vertex of the alcove corresponding to the simple root $\alpha_1$; this root is the terminal long root in the nonaffine Dynkin diagram.  Choose $\bar{D} \bs \alpha_1$ as the positive simple roots of the centralizer $Z_e$.  The $Z_e$-positive Weyl group is
\begin{equation} \nn
\begin{split}
W^{{Z_e}_+} = \{ 
& 1, w_1, w_2 w_1, w_3 w_2 w_1, \ldots, w_n w_{n-1} \cdots w_1,
\\
& w_{n-1} w_n w_{n-1} \cdots w_1, w_{n-2} w_{n-1} w_n w_{n-1} \cdots w_1, \ldots, w_1 w_2 \cdots w_n w_{n-1} \cdots w_1\}
\end{split}
\end{equation}
The corresponding collection of $Z_e$-positive weights is
\begin{equation} \nn
\Lambda^{{Z_e}_+} = \{0, \lam_1, \lam_2, \lam_3, \ldots, \lam_n, \lam_{n-1}, \lam_{n-2}, \ldots, \lam_1\}
\end{equation}
and the resulting Steinberg basis is
\begin{equation} \nn
\begin{split}
B_e = \{
& 0, w_1 \lam_1, w_1 w_2 \lam_2, w_1 w_2 w_3 \lam_3, \ldots, w_1 w_2 \cdots w_n \lam_n, 
\\
& w_1 w_2 \cdots w_n w_{n-1} \lam_{n-1}, \ldots, w_1 w_2 \cdots w_n w_{n-1} \cdots w_1 \lam_1
\}
\end{split}
\end{equation}
The centralizer $Z_v$ is the group $G$, and therefore $B_v = \{0\}$.

Let $r$ denote the automorphism of $\ft^{\ast}$ determined by the nontrivial order $2$ automorphism of the affine Dynkin diagram that exchanges $\alpha_1$ and $\alpha_0$.  This automorphism preserves the positive root system of $Z_e$ while transforming the positive simple roots $\bar{D}$ of $Z_v$ into the roots $D \bs \alpha_1$.  The pair
\[(AB_e, AB_v) = (r(B_e) + k \lam_1, r(B_v) + k \lam_1)\]
is a level $k$ affine Steinberg basis for $(Z_e, Z_v)$.  

The single weight $AB_v = \{k \lam_1\}$ is level $k$.  The remaining weights of $AB_e$ are
\begin{equation} \nn
\begin{split}
AB_e \bs AB_v = \{ 
&
(k\!-\!1)\lam_1 + \lam_2,
(k\!-\!1)\lam_1 + \lam_3,
\ldots,
(k\!-\!1)\lam_1 + \lam_{n-1},
(k\!-\!1)\lam_1 + 2 \lam_n,
\\
&
k \lam_1 + \lam_n,
k \lam_1 + \lam_{n-1},
\ldots,
k \lam_1 + \lam_2,
(k\!+\!1) \lam_1
\}
\end{split}
\end{equation}
The weight $k \lam_1 + \lam_2 \in AB_e \bs AB_v$ is level $k+2$ and induces to the weight $k \lam_1 \in AB_v$.  The first $n$ weights and the very last weight of $AB_e \bs AB_v$, in the above order, are level $k+1$, therefore singular.  After the first $n$ weights, the next $n-3$ weights are level $k+2$ but are nevertheless singular.  The affine Steinberg basis is therefore induction closed.  By Proposition~\ref{prop-indclosed} the fusion kernel is 
$$\big\langle\big\{[\lam]_e \st \lam \in AB_e \bs (AB_v \cup \{k \lam_1 + \lam_2\})\big\} \cup \big\{[k \lam_1 + \lam_2]_e + [k \lam_1]_e\big\}\big\rangle$$ 
and the resulting fusion ideal is listed in Theorem~\ref{thm-fusion}.

\subsubsection*{Type C} \label{sec-c}

Let $v$ be the vertex of the alcove corresponding to the simple root $\alpha_1$; this root is the terminal short root in the nonaffine Dynkin diagram.  Choose $\bar{D} \bs \alpha_1$ as the positive simple roots of the centralizer $Z_e$.  The $Z_e$-positive Weyl group is
\begin{equation} \nn
\begin{split}
W^{{Z_e}_+} = \{
& 1, w_1, w_2 w_1, w_3 w_2 w_1, \ldots, w_n w_{n-1} \cdots w_1,
\\
& w_{n-1} w_n w_{n-1} \cdots w_1, w_{n-2} w_{n-1} w_n w_{n-1} \cdots w_1, \ldots, w_1 w_2 \cdots w_n w_{n-1} \cdots w_1\}
\end{split}
\end{equation}
The corresponding collection of $Z_e$-positive weights is
\begin{equation} \nn
\Lambda^{{Z_e}_+} = \{0, \lam_1, \lam_2, \lam_3, \ldots, \lam_n, \lam_{n-1}, \lam_{n-2}, \ldots, \lam_1\}
\end{equation}
and the resulting Steinberg basis is
\begin{equation} \nn
\begin{split}
B_e = \{
& 0, w_1 \lam_1, w_1 w_2 \lam_2, w_1 w_2 w_3 \lam_3, \ldots, w_1 w_2 \cdots w_n \lam_n,
\\
& w_1 w_2 \cdots w_n w_{n-1} \lam_{n-1}, \ldots, w_1 w_2 \cdots w_n w_{n-1} \cdots w_1 \lam_1
\}
\end{split}
\end{equation}
Choose $(\bar{D} \cup {-}\alpha_0) \bs \alpha_1$ as the positive simple roots of the centralizer $Z_v$.  The $Z_v$-positive Weyl group elements are the first $n$ $Z_e$-positive elements:
\begin{equation} \nn
W^{{Z_v}_+} = \{1, w_1, w_2 w_1, w_3 w_2 w_1, \ldots, w_{n-1} w_{n-2} \cdots w_1\}
\end{equation}
The resulting Steinberg basis is
\begin{equation} \nn
B_v = \{0, w_1 \lam_1, w_1 w_2 \lam_2, w_1 w_2 w_3 \lam_3, \ldots, w_1 w_2 \cdots w_{n-1} \lam_{n-1}\}
\end{equation}

Let $r$ denote the Weyl reflection $w_0$.  This automorphism fixes the positive root system of $Z_e$ while transforming the positive simple roots $(\bar{D} \cup -\alpha_0) \bs \alpha_1$ of $Z_v$ into the roots $D \bs \alpha_1$.  The pair
\[(AB_e, AB_v) = (r(B_e) + k \lam_1, r(B_v) + k \lam_1)\]
is a level $k$ affine Steinberg basis for $(Z_e, Z_v)$.  

The weights of the complement $AB_e \bs AB_v$ are all level $k+1$, thus singular.  The affine Steinberg basis is therefore induction closed.  By Proposition~\ref{prop-indclosed} the fusion kernel is $\{[\lam]_e \st \lam \in AB_e \bs AB_v\}$, and the resulting fusion ideal is listed in Theorem~\ref{thm-fusion}.

\subsubsection*{Type D} \label{sec-d}

Let $v$ be the vertex of the alcove corresponding to the simple root $\alpha_1$; this root is the isolated terminal root in the nonaffine Dynkin diagram.  Choose $\bar{D} \bs \alpha_1$ as the positive simple roots of the centralizer $Z_e$.  The $Z_e$-positive Weyl group is
\begin{equation} \nn
\begin{split}
W^{{Z_e}_+} = \{
& 1, w_1, w_2 w_1, w_3 w_2 w_1, 
\ldots, 
w_{n-1} w_{n-2} \cdots w_1, 
w_n w_{n-2} w_{n-3} \cdots w_1, 
w_n w_{n-1} \cdots w_1,
\\
& w_{n-2} w_n w_{n-1} \cdots w_1, 
w_{n-3} w_{n-2} w_n w_{n-1} \cdots w_1, 
\ldots, 
w_1 w_2 \cdots w_{n-2} w_n w_{n-1} \cdots w_1\}
\end{split}
\end{equation}
The corresponding collection of $Z_e$-positive weights is
\begin{equation} \nn
\Lambda^{{Z_e}_+} = \{0, \lam_1, \lam_2, \lam_3, 
\ldots, 
\lam_{n-1}, 
\lam_n, 
\lam_{n-1} + \lam_n, 
\lam_{n-2}, 
\lam_{n-3}, 
\ldots, 
\lam_1\}
\end{equation}
and the resulting Steinberg basis is
\begin{equation} \nn
\begin{split}
B_e = \{
& 
0, w_1 \lam_1, w_1 w_2 \lam_2, w_1 w_2 w_3 \lam_3, 
\ldots, 
w_1 w_2 \cdots w_{n-1} \lam_{n-1}, 
w_1 w_2 \cdots w_{n-2} w_n \lam_n,
\\
&
w_1 w_2 \cdots w_n (\lam_{n-1} + \lam_n),
w_1 w_2 \cdots w_n w_{n-2} \lam_{n-2},
\\
&
w_1 w_2 \cdots w_n w_{n-2} w_{n-3} \lam_{n-3}, 
\ldots, 
w_1 w_2 \cdots w_n w_{n-2} w_{n-3} \cdots w_1 \lam_1
\}
\end{split}
\end{equation}
The centralizer $Z_v$ is the group $G$, and therefore $B_v = \{0\}$.

Let $r$ denote the automorphism of $\ft^{\ast}$ determined by the nontrivial order $2$ automorphism of the affine Dynkin diagram that exchanges $\alpha_1$ and $\alpha_0$, and fixes the other simple roots.  This automorphism preserves the positive root system of $Z_e$ while transforming the positive simple roots $\bar{D}$ of $Z_v$ into the roots $D \bs \alpha_1$.  The pair
\[(AB_e, AB_v) = (r(B_e) + k \lam_1, r(B_v) + k \lam_1)\]
is a level $k$ affine Steinberg basis for $(Z_e, Z_v)$.  

The single weight $AB_v = \{k \lam_1\}$ is level $k$.  The remaining weights of $AB_e$ are
\begin{equation} \nn
\begin{split}
AB_e \bs AB_v = \{
&
(k\!-\!1)\lam_1 + \lam_2,
(k\!-\!1)\lam_1 + \lam_3,
\ldots,
(k\!-\!1)\lam_1 + \lam_{n-2},
\\
&
(k\!-\!1)\lam_1 + \lam_{n-1} + \lam_n,
k \lam_1 + \lam_n,
k \lam_1 + \lam_{n-1},
k \lam_1 + \lam_{n-1} + \lam_n,
\\
&
k \lam_1 + \lam_{n-2},
k \lam_1 + \lam_{n-3},
\ldots,
k \lam_1 + \lam_2,
(k\!+\!1) \lam_1
\}
\end{split}
\end{equation}
The weight $k \lam_1 + \lam_2 \in AB_e \bs AB_v$ is level $k+2$ and induces to the weight $k \lam_1 \in AB_v$.  The first $n$ weights and the very last weight of $AB_e \bs AB_v$, in the above order, are level $k+1$, therefore singular.  After the first $n$ weights, the next $n-3$ weights are level $k+2$ but are nevertheless singular.  The affine Steinberg basis is therefore induction closed.  By Proposition~\ref{prop-indclosed} the fusion kernel is 
$$\big\langle\big\{[\lam]_e \st \lam \in AB_e \bs (AB_v \cup \{k \lam_1 + \lam_2\})\big\} \cup \big\{[k \lam_1 + \lam_2]_e + [k \lam_1]_e\big\}\big\rangle$$ 
and the resulting fusion ideal is listed in Theorem~\ref{thm-fusion}.

\subsubsection*{$G_2$} \label{sec-g}

Let $v$ be the vertex of the alcove corresponding to the simple root $\alpha_{2}$; this root is the long root in the nonaffine Dynkin diagram.  Choose $\bar{D} \bs \alpha_{2}$ as the positive simple roots of the centralizer $Z_e$.  The $Z_e$-positive Weyl group is
\begin{equation} \nn
W^{{Z_e}_+} = \{1, w_{2}, w_{1} w_{2}, w_{2} w_{1} w_{2}, w_{1} w_{2} w_{1} w_{2}, w_{2} w_{1} w_{2} w_{1} w_{2} \}
\end{equation}
The corresponding collection of $Z_e$-positive weights is
\begin{equation} \nn
\Lambda^{{Z_e}_+} = \{0, \lam_{2}, \lam_{1}, \lam_{2}, \lam_{1}, \lam_{2}\}
\end{equation}
and the resulting Steinberg basis is
\begin{equation} \nn
B_e = \{0, w_{2} \lam_{2}, w_{2} w_{1} \lam_{1}, w_{2} w_{1} w_{2} \lam_{2}, w_{2} w_{1} w_{2} w_{1} \lam_{1}, 
w_{2} w_{1} w_{2} w_{1} w_{2} \lam_{2}
\}
\end{equation}
Choose $(\bar{D} \cup -\alpha_0) \bs \alpha_{2}$ as the positive simple roots of the centralizer $Z_v$.  The $Z_v$-positive Weyl group elements are the first $3$ $Z_e$-positive elements:
\begin{equation} \nn
W^{{Z_v}_+} = \{1, w_{2}, w_{1} w_{2}\}
\end{equation}
The resulting Steinberg basis is
\begin{equation} \nn
B_v = \{0, w_{2} \lam_{2}, w_{2} w_{1} \lam_{1}\}
\end{equation}

Suppose the level $k$ is even.
Let $r$ denote the Weyl reflection $w_0$.  This automorphism fixes the positive root system of $Z_e$ while transforming the positive simple roots $(\bar{D} \cup -\alpha_0) \bs \alpha_{2}$ of $Z_v$ into the roots $D \bs \alpha_{2}$.  The pair
\[(AB^{\text{ev}}_e, AB^{\text{ev}}_v) = \Big(r(B_e) + k \frac{\lam_{2}}{2}, r(B_v) + k \frac{\lam_{2}}{2}\Big)\]
is a level $k$ affine Steinberg basis for $(Z_e, Z_v)$.  Of the three weights 
\[AB^{\text{ev}}_e \bs AB^{\text{ev}}_v = 
\Big\{\Big(\frac{k}{2} - 1\Big) \lam_{2} + 3 \lam_{1}, \frac{k}{2} \lam_{2} + \lam_{1}, \Big(\frac{k}{2} + 1\Big) \lam_{2}\Big\}\]
the first two are level $k+1$ thus singular, and the third induces to the weight $\frac{k}{2} \lam_{2}$.  The affine Steinberg basis is therefore induction closed.  Proposition~\ref{prop-indclosed} identifies the fusion kernel, and Lemma~\ref{lemma-onevert} the fusion ideal---the result is listed in Theorem~\ref{thm-fusion}.

Suppose the level $k$ is odd.
Instead of a reflection, we apply an affine transformation to the Steinberg basis.  The pair
\[(AB^{\text{odd}}_e, AB^{\text{odd}}_v) = \Big(B_e + (k+1) \frac{\lam_{2}}{2}, B_v + (k+1) \frac{\lam_{2}}{2}\Big)\]
is a level $k$ affine Steinberg basis for $(Z_e, Z_v)$.  That $AB^{\text{odd}}_e$ provides an $R[G]$-basis for $R[Z_e]$ is clear, so it suffices to check that 
\[
AB^{\text{odd}}_v = \Big\{\frac{k-3}{2} \lam_{2} + 3 \lam_{1}, \frac{k-1}{2} \lam_{2} + \lam_{1}, \frac{k-1}{2} \lam_{2}\Big\}
\]
provides an $R[G]$-basis for $R_k[Z_v]$.  First note that the span of the irreducible twisted representations $\{[\lam]_v \st \lam \in AB^{\text{odd}}_v\}$ contains the twisted representation $[\frac{k-3}{2} \lam_{2} + 2 \lam_{1}]_v$; next note that the span of $\{[\frac{k-1}{2} \lam_{2} + \lam_{1}]_v, [\frac{k-1}{2} \lam_{2}]_v, [\frac{k-3}{2} \lam_{2} + 2 \lam_{1}]_v\}$ contains $[\frac{k-3}{2} \lam_{2} + \lam_{1}]_v$; finally note that Lemma 3.2 in~\cite{douglas-sfi} proves that $\{[\frac{k-1}{2} \lam_{2} + \lam_{1}]_v, [\frac{k-1}{2} \lam_{2}]_v, [\frac{k-3}{2} \lam_{2} + \lam_{1}]_v\}$ is an $R[G]$-basis for $R_k[Z_v]$.  The affine Steinberg basis $(AB^{\text{odd}}_e, AB^{\text{odd}}_v)$ is certainly induction closed.  As before Proposition~\ref{prop-indclosed} and Lemma~\ref{lemma-onevert} compute the fusion ideal and the result is listed in Theorem~\ref{thm-fusion}.

\subsection{The exceptional groups $E_6$ and $E_7$} \label{sec-e67}

\subsubsection*{$E_6$}

Let $v$ be the vertex of the alcove corresponding to the simple root $\alpha_1$; this root is a terminal root of a long leg of the nonaffine Dynkin diagram.  Choose $\bar{D} \bs \alpha_1$ as the positive simple roots of the centralizer $Z_e$.  The $Z_e$-positive Weyl group is
\begin{equation} \nn
\begin{split}
W^{{Z_e}_+} = 
\{
&
\mathrm{id}, 1, 3 1, 4 3 1, 2 4 3 1, 5 4 3 1, 5 2 4 3 1, 6 5 4 3 1, 4 5 2 4 3 1, 2 6 5 4 3 1, 3 4 5 2 4 3 1, 4 2 6 5 4 3 1, 1 3 4 5 2 4 3 1, 
\\
&
6 3 4 5 2 4 3 1, 5 4 2 6 5 4 3 1, 6 1 3 4 5 2 4 3 1, 3 5 4 2 6 5 4 3 1, 5 6 1 3 4 5 2 4 3 1, 4 3 5 4 2 6 5 4 3 1, 4 5 6 1 3 4 5 2 4 3 1, 
\\
&
2 4 3 5 4 2 6 5 4 3 1, 3 4 5 6 1 3 4 5 2 4 3 1, 1 2 4 3 5 4 2 6 5 4 3 1, 2 3 4 5 6 1 3 4 5 2 4 3 1, 4 2 3 4 5 6 1 3 4 5 2 4 3 1, 
\\
&
5 4 2 3 4 5 6 1 3 4 5 2 4 3 1, 6 5 4 2 3 4 5 6 1 3 4 5 2 4 3 1
\}
\end{split}
\end{equation}
Here we have abbreviated the Weyl element $w_i$ as $i$.
The corresponding collection of $Z_e$-positive weights is
\begin{equation} \nn
\begin{split}
\Lambda^{{Z_e}_+} = \{
&
\lam_w \st w \in W^{{Z_e}_+}\} \\
=
\{
&
0, \lam_1, \lam_3, \lam_4, \lam_2, \lam_5, \lam_2 + \lam_5, \lam_6, \lam_4, \lam_2 + \lam_6, \lam_3, \lam_4 + \lam_6, \lam_1, \lam_3 + \lam_6, \lam_5, \lam_1 + \lam_6, \\
&
\lam_3 + \lam_5, \lam_1 +\lam_5, \lam_4, \lam_1 + \lam_4, \lam_2, \lam_3, \lam_1 + \lam_2, \lam_2 + \lam_3, \lam_4, \lam_5, \lam_6
\}
\end{split}
\end{equation}
The resulting Steinberg basis is $B_e = \{w^{-1} \lam_w \st w \in W^{{Z_e}_+}\}$.  The centralizer $Z_v$ is the group $G$, and therefore $B_v = \{0\}$.

Let $r$ denote the automorphism of $\ft^{\ast}$ determined by the nontrivial order $2$ automorphism of the affine Dynkin diagram that exchanges $\alpha_1$ and $\alpha_0$.  This automorphism preserves the positive root system of $Z_e$ while transforming the positive simple roots $\bar{D}$ of $Z_v$ into the roots $D \bs \alpha_1$.  The pair
\[(AB_e, AB_v) = (r(B_e) + k \lam_1, r(B_v) + k \lam_1)\]
is a level $k$ affine Steinberg basis for $(Z_e, Z_v)$.  The weights of the complement $AB_e \bs AB_v$ are
\begin{equation} \nn
\begin{split}
AB_e \bs AB_v =
\{
&
[k-1,1,0,0,0,0],
[k-2,0,0,1,0,0],
[k-3,0,1,0,1,0],
[k-1,0,0,0,1,0],
\\
&
[k-2,0,1,0,0,1],
[k-3,0,1,0,1,1],
[k-1,0,1,0,0,0],
[k-2,0,0,1,0,1],
\\
&
[k-2,0,1,0,1,0],
[k-1,1,0,0,0,1],
[k-3,0,1,1,0,1],
[k,0,0,0,0,1],
\\
&
[k-2,1,1,0,0,1],
[k-1,0,0,1,0,0],
[k-1,0,1,0,0,1],
[k-2,1,0,1,0,1],
\\
&
[k-1,0,0,1,0,1],
[k-1,1,0,0,1,0],
[k-1,1,0,0,1,1],
[k,1,0,0,0,0],
\\
&
[k,0,0,0,1,0],
[k,1,0,0,0,1],
[k,1,0,0,1,0],
[k,0,0,1,0,0],
[k,0,1,0,0,0],
\\
&
[k+1,0,0,0,0,0]
\}
\end{split}
\end{equation}
Here $[n_1, n_2, n_3, n_4, n_5, n_6]$ denotes the weight $\sum n_i \lam_i$.  We will refer to those twenty-six weights in the order listed.

Because the vertex $v$ is central, Proposition~\ref{prop-central} describes the fusion kernel.  Every weight of $AB_e \bs AB_v$ is singular for induction to $R_k[Z_v]$, except weights 10, 11, 16, 17, 20, and 24.  Those six weights induce to, respectively, the signed weights
\begin{equation} \nn
\begin{split}
&
{-}[k-1,0,0,0,0,1],
+[k-3,0,1,0,0,1],
+[k-2,0,0,0,0,1],
\\
&
{+}[k-1,0,0,0,0,1],
-[k,0,0,0,0,0],
+[k,0,0,0,0,0]
\end{split}
\end{equation}
Let $w \in W$ be such that $w(C_v - k \lam_1) = C$---this transformation $\lam \mapsto w(\lam - k \lam_1)$ associates highest weights for representations in $R[G]$ to highest weights for representations in $R_k[Z_v]$.  Under this transformation, the above six signed weights become respectively
\begin{equation} \nn
\begin{split}
&
{-}[1,0,0,0,0,0],
+[1,1,0,0,0,0],
+[1,0,0,0,0,1],
\\
&
{+}[1,0,0,0,0,0],
-[0,0,0,0,0,0],
+[0,0,0,0,0,0]
\end{split}
\end{equation}
By Proposition~\ref{prop-central} and Lemma~\ref{lemma-onevert}, the resulting fusion ideal is as listed in Theorem~\ref{thm-fusion}.

\subsubsection*{$E_7$}

Let $v$ be the vertex of the alcove corresponding to the simple root $\alpha_7$; this root is the terminal root of the long leg of the nonaffine Dynkin diagram.  Choose $\bar{D} \bs \alpha_7$ as the positive simple roots of the centralizer $Z_e$.  The $Z_e$-positive Weyl group is
\begin{equation} \nn
\begin{split}
W^{{Z_e}_+} = 
\{
&
\mathrm{id}, 7, 6 7, 5 6 7, 4 5 6 7, 2 4 5 6 7, 3 4 5 6 7, 3 2 4 5 6 7, 1 3 4 5 6 7, 4 3 2 4 5 6 7, 2 1 3 4 5 6 7, 5 4 3 2 4 5 6 7, 
\\
&
1 4 3 2 4 5 6 7, 6 5 4 3 2 4 5 6 7, 1 5 4 3 2 4 5 6 7, 3 1 4 3 2 4 5 6 7, 7 6 5 4 3 2 4 5 6 7, 1 6 5 4 3 2 4 5 6 7, 
\\
&
3 1 5 4 3 2 4 5 6 7, 1 7 6 5 4 3 2 4 5 6 7, 3 1 6 5 4 3 2 4 5 6 7, 4 3 1 5 4 3 2 4 5 6 7, 3 1 7 6 5 4 3 2 4 5 6 7, 
\\
&
4 3 1 6 5 4 3 2 4 5 6 7, 2 4 3 1 5 4 3 2 4 5 6 7, 4 3 1 7 6 5 4 3 2 4 5 6 7, 5 4 3 1 6 5 4 3 2 4 5 6 7, 2 4 3 1 6 5 4 3 2 4 5 6 7, 
\\
&
5 4 3 1 7 6 5 4 3 2 4 5 6 7, 2 4 3 1 7 6 5 4 3 2 4 5 6 7, 5 2 4 3 1 6 5 4 3 2 4 5 6 7, 6 5 4 3 1 7 6 5 4 3 2 4 5 6 7, 
\\
&
2 5 4 3 1 7 6 5 4 3 2 4 5 6 7, 4 5 2 4 3 1 6 5 4 3 2 4 5 6 7, 2 6 5 4 3 1 7 6 5 4 3 2 4 5 6 7, 4 2 5 4 3 1 7 6 5 4 3 2 4 5 6 7, 
\\
&
3 4 5 2 4 3 1 6 5 4 3 2 4 5 6 7, 4 2 6 5 4 3 1 7 6 5 4 3 2 4 5 6 7, 3 4 2 5 4 3 1 7 6 5 4 3 2 4 5 6 7, 1 3 4 5 2 4 3 1 6 5 4 3 2 4 5 6 7, 
\\
&
5 4 2 6 5 4 3 1 7 6 5 4 3 2 4 5 6 7, 3 4 2 6 5 4 3 1 7 6 5 4 3 2 4 5 6 7, 1 3 4 2 5 4 3 1 7 6 5 4 3 2 4 5 6 7, 
\\
&
3 5 4 2 6 5 4 3 1 7 6 5 4 3 2 4 5 6 7, 1 3 4 2 6 5 4 3 1 7 6 5 4 3 2 4 5 6 7, 4 3 5 4 2 6 5 4 3 1 7 6 5 4 3 2 4 5 6 7, 
\\
&
1 3 5 4 2 6 5 4 3 1 7 6 5 4 3 2 4 5 6 7, 2 4 3 5 4 2 6 5 4 3 1 7 6 5 4 3 2 4 5 6 7, 1 4 3 5 4 2 6 5 4 3 1 7 6 5 4 3 2 4 5 6 7, 
\\
&
1 2 4 3 5 4 2 6 5 4 3 1 7 6 5 4 3 2 4 5 6 7, 3 1 4 3 5 4 2 6 5 4 3 1 7 6 5 4 3 2 4 5 6 7, 3 1 2 4 3 5 4 2 6 5 4 3 1 7 6 5 4 3 2 4 5 6 7, 
\\
&
4 3 1 2 4 3 5 4 2 6 5 4 3 1 7 6 5 4 3 2 4 5 6 7, 5 4 3 1 2 4 3 5 4 2 6 5 4 3 1 7 6 5 4 3 2 4 5 6 7, 
\\
&
6 5 4 3 1 2 4 3 5 4 2 6 5 4 3 1 7 6 5 4 3 2 4 5 6 7, 7 6 5 4 3 1 2 4 3 5 4 2 6 5 4 3 1 7 6 5 4 3 2 4 5 6 7
\}
\end{split}
\end{equation}
The corresponding collection of $Z_e$-positive weights is
\begin{equation} \nn
\begin{split}
\Lambda^{{Z_e}_+} = \{
&
0,\lam_7,\lam_6,\lam_5,\lam_4,\lam_2,\lam_3,\lam_2 + \lam_3,\lam_1,\lam_4,\lam_1 + \lam_2,\lam_5,\lam_1 + \lam_4,\lam_6,\lam_1 + \lam_5,\lam_3,\lam_7,
\\
&
\lam_1 + \lam_6,\lam_3 + \lam_5,\lam_1 + \lam_7,\lam_3 + \lam_6,\lam_4,\lam_3 + \lam_7,\lam_4 + \lam_6,\lam_2,\lam_4 + \lam_7,\lam_5,\lam_2 + \lam_6,
\\
&
\lam_5 + \lam_7,\lam_2 + \lam_7,\lam_2 + \lam_5,\lam_6,\lam_2 + \lam_5 + \lam_7,\lam_4,\lam_2 + \lam_6,\lam_4 + \lam_7,\lam_3,\lam_4 + \lam_6,
\\
&
\lam_3 + \lam_7,\lam_1,\lam_5,\lam_3 + \lam_6,\lam_1 + \lam_7,\lam_3 + \lam_5,\lam_1 + \lam_6,\lam_4,\lam_1 + \lam_5,\lam_2,\lam_1 + \lam_4,
\\
&
\lam_1 + \lam_2,\lam_3,\lam_2 + \lam_3,\lam_4,\lam_5,\lam_6,\lam_7
\}
\end{split}
\end{equation}
The resulting Steinberg basis is $B_e = \{w^{-1} \lam_w \st w \in W^{{Z_e}_+}\}$.  The centralizer $Z_v$ is the group $G$, and therefore $B_v = \{0\}$.

Let $r$ denote the automorphism of $\ft^{\ast}$ determined by the nontrivial order $2$ automorphism of the affine Dynkin diagram that exchanges $\alpha_7$ and $\alpha_0$.  This automorphism preserves the positive root system of $Z_e$ while transforming the positive simple roots $\bar{D}$ of $Z_v$ into the roots $D \bs \alpha_7$.  The pair
\[(AB_e, AB_v) = (r(B_e) + k \lam_7, r(B_v) + k \lam_7)\]
is a level $k$ affine Steinberg basis for $(Z_e, Z_v)$.  The weights of the complement $AB_e \bs AB_v$ are
\begin{equation} \nn
\begin{split}
AB_e \bs AB_v =
\{
&
[1,0,0,0,0,0,k-1],
[0,0,1,0,0,0,k-2],
[0,0,0,1,0,0,k-3],
[0,1,0,0,1,0,k-4],
\\
&
[0,0,0,0,1,0,k-2],
[0,1,0,0,0,1,k-3],
[0,1,0,0,1,1,k-5],
[0,1,0,0,0,0,k-1],
\\
&
[0,0,0,1,0,1,k-4],
[0,1,0,0,1,0,k-3],
[0,0,1,0,0,1,k-3],
[0,1,0,1,0,1,k-5],
\\
&
[1,0,0,0,0,1,k-2],
[0,1,1,0,0,1,k-4],
[0,0,0,1,0,0,k-2],
[0,0,0,0,0,1,k-1],
\\
&
[1,1,0,0,0,1,k-3],
[0,0,1,1,0,1,k-5],
[0,1,0,0,0,1,k-2],
[1,0,0,1,0,1,k-4],
\\
&
[0,0,1,0,1,0,k-3],
[0,0,0,1,0,1,k-3],
[1,0,1,0,1,1,k-5],
[0,0,1,0,0,0,k-1],
\\
&
[0,0,1,0,1,1,k-4],
[1,0,0,0,1,0,k-2],
[1,0,1,0,0,1,k-3],
[1,0,0,0,1,1,k-3],
\\
&
[0,0,1,0,0,1,k-2],
[1,0,1,0,1,0,k-3],
[0,0,0,0,1,0,k-1],
[1,0,1,0,1,1,k-4],
\\
&
[1,0,0,1,0,0,k-2],
[0,0,1,0,1,0,k-2],
[1,0,0,1,0,1,k-3],
[1,1,0,0,0,0,k-1],
\\
&
[1,0,0,1,1,0,k-3],
[1,1,0,0,0,1,k-2],
[1,0,0,0,0,0,k],
[0,0,0,1,0,0,k-1],
\\
&
[1,1,0,0,1,0,k-2],
[1,0,0,0,0,1,k-1],
[1,1,0,1,0,0,k-2],
[1,0,0,0,1,0,k-1],
\\
&
[0,1,1,0,0,0,k-1],
[1,0,0,1,0,0,k-1],
[0,1,0,0,0,0,k],
[1,1,1,0,0,0,k-1],
\\
&
[1,1,0,0,0,0,k],
[0,0,1,0,0,0,k],
[0,1,1,0,0,0,k],
[0,0,0,1,0,0,k],
\\
&
[0,0,0,0,1,0,k],
[0,0,0,0,0,1,k],
[0,0,0,0,0,0,k+1]
\}
\end{split}
\end{equation}
Here $[n_1, n_2, n_3, n_4, n_5, n_6, n_7]$ denotes the weight $\sum n_i \lam_i$.  We will refer to those fifty-five weights in the order listed.

Because the vertex $v$ is central, Proposition~\ref{prop-central} describes the fusion kernel.  Every weight of $AB_e \bs AB_v$ is singular for induction to $R_k[Z_v]$, except weights 13, 14, 21, 27, 29, 35, 39, 46, 50, and 52.  Those ten weights induce to, respectively, the signed weights
\begin{equation} \nn
\begin{split}
&
{-}[0,0,0,0,0,1,k-2],
+[0,1,0,0,0,1,k-4],
+[0,0,0,0,1,0,k-3],
+[0,0,0,0,0,1,k-3],
\\
&
{+}[0,0,0,0,0,1,k-2],
-[0,0,0,0,0,1,k-3],
-[0,0,0,0,0,0,k],
-[0,0,0,0,0,0,k-1],
\\
&
{+}[0,0,0,0,0,0,k],
-[0,0,0,0,0,0,k]
\end{split}
\end{equation}
Let $w \in W$ be such that $w(C_v - k \lam_7) = C$---this transformation $\lam \mapsto w(\lam - k \lam_7)$ associates highest weights for representations in $R[G]$ to highest weights for representations in $R_k[Z_v]$.  Under this transformation, the above ten signed weights become respectively
\begin{equation} \nn
\begin{split}
&
{-}[1,0,0,0,0,0,0],
+[1,1,0,0,0,0,0],
+[0,0,1,0,0,0,0],
+[1,0,0,0,0,0,1],
\\
&
{+}[1,0,0,0,0,0,0],
-[1,0,0,0,0,0,1],
-[0,0,0,0,0,0,0],
-[0,0,0,0,0,0,1],
\\
&
{+}[0,0,0,0,0,0,0],
-[0,0,0,0,0,0,0]
\end{split}
\end{equation}
By Proposition~\ref{prop-central} and Lemma~\ref{lemma-onevert}, the resulting fusion ideal is as listed in Theorem~\ref{thm-fusion}.

\subsection{The exceptional groups $F_4$ and $E_8$} \label{sec-fe}

\subsubsection*{$F_4$}

Let $v$ be the vertex of the alcove corresponding to the simple root $\alpha_1$; this root is the terminal long root in the nonaffine Dynkin diagram.  Choose $\bar{D} \bs \alpha_1$ as the positive simple roots of the centralizer $Z_e = \langle Sp(3) \rangle$.  Here, and henceforth, when there is an implicit inclusion of Lie algebras $\mathfrak{h} \subset \mathfrak{g}$, the notation $\langle H \rangle$ refers to the full rank subgroup of $G$ whose Lie algebra has semisimple part the Lie algebra $\mathfrak{h}$ of $H$.  The $Z_e$-positive Weyl group is
\begin{equation} \nn
\begin{split}
W^{{Z_e}_+} = \{
&
\mathrm{id}, 1, 2 1, 3 2 1, 4 3 2 1, 2 3 2 1, 2 4 3 2 1, 1 2 3 2 1, 3 2 4 3 2 1, 1 2 4 3 2 1, 2 3 2 4 3 2 1, 
\\
&
1 3 2 4 3 2 1, 1 2 3 2 4 3 2 1, 2 1 3 2 4 3 2 1, 2 1 2 3 2 4 3 2 1, 3 2 1 3 2 4 3 2 1, 3 2 1 2 3 2 4 3 2 1, 
\\
&
4 3 2 1 3 2 4 3 2 1, 2 3 2 1 2 3 2 4 3 2 1, 1 4 3 2 1 3 2 4 3 2 1, 4 2 3 2 1 2 3 2 4 3 2 1, 
\\
&
3 4 2 3 2 1 2 3 2 4 3 2 1, 2 3 4 2 3 2 1 2 3 2 4 3 2 1, 1 2 3 4 2 3 2 1 2 3 2 4 3 2 1
\}
\end{split}
\end{equation}
The corresponding collection of $Z_e$-positive weights is
\begin{equation} \nn
\begin{split}
\Lambda^{{Z_e}_+} = \{
&
0,\lam_1,\lam_2,\lam_3,\lam_4,\lam_2,\lam_2 + \lam_4,\lam_1,\lam_3,\lam_1 + \lam_4,\lam_2,\lam_1 + \lam_3,\lam_1,\lam_2,
\\
&
\lam_1 + \lam_2,\lam_3,\lam_1 + \lam_3,\lam_4,\lam_2,\lam_1 + \lam_4,\lam_2 + \lam_4,\lam_3,\lam_2,\lam_1
\} 
\end{split}
\end{equation}
and the resulting Steinberg basis is $B_e = \{w^{-1} \lam_w \st w \in W^{{Z_e}_+}\}$.
Choose $(\bar{D} \cup -\alpha_0) \bs \alpha_1$ as the positive simple roots of the centralizer $Z_v$.  The $Z_v$-positive Weyl group elements are the first $12$ $Z_e$-positive elements:
\begin{equation} \nn
W^{{Z_v}_+} = \{\mathrm{id}, 1, 2 1, 3 2 1, 4 3 2 1, 2 3 2 1, 2 4 3 2 1, 1 2 3 2 1, 3 2 4 3 2 1, 1 2 4 3 2 1, 2 3 2 4 3 2 1, 1 3 2 4 3 2 1\}
\end{equation}
The vertex Steinberg basis is then $B_v = \{w^{-1} \lam_w \st w \in W^{{Z_v}_+}\}$.

As a warmup for the case of $E_8$, we compress the above information about the positive Weyl groups and positive weights into Figure~\ref{fig-f4}.  The rightmost edges in the figure are labelled by numbers $i$.  The remaining edges can be numbered as follows: if translation of an edge in a direction perpendicular to that edge produces another edge of the figure, those edges share a label; the labels of the central hexagon alternate $1$, $2$, $1$, $2$, $1$, $2$.  Each node $p$ of the figure represents an element of the $Z_e$-positive Weyl group $W^{{Z_e}_+}$, namely the element $w_p := w_{i_1} w_{i_2} \cdots w_{i_m}$ where $i_1, i_2, \ldots, i_m$ is the sequence of labels along a path from that node $p$ to the top node of the figure.  The $Z_e$-positive weight $\lam_{w_p}$ corresponding to the $Z_e$-positive Weyl element $w_p$ is $\lam_{w_p} = \sum_{i \searrow p} \lam_i$; this sum is over the labels on the edges directly descending to the node $p$.  The $Z_v$-positive Weyl group is the collection of Weyl elements associated to the nodes in the top half of the diagram.

\begin{figure}[htb]
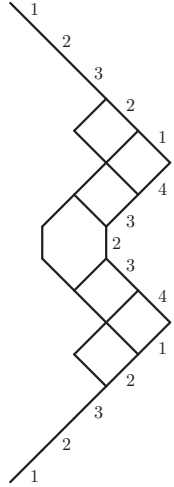

\begin{center}
\ingm{f4fig}
\caption{The $\langle Sp(3) \rangle$-positive Weyl group of $F_4$} \label{fig-f4}
\end{center}
\end{figure}

Suppose the level $k$ is even.  Let $r$ denote the Weyl reflection $w_0$.  The collection of weights
\[AB^{\text{ev}}_e = r(B_e) + k \frac{\lam_1}{2}\]
is, like $B_e$, a Steinberg basis for $Z_e$.  (It is moreover a level $k$ affine Steinberg basis for the pair $(Z_e, Z_v)$, but we will not exploit this fact.)  The weights of $AB^{\text{ev}}_e$ are as follows:
\begin{equation} \nn
\begin{split}
AB^{\text{ev}}_e = \bigg\{
&
\left[\frac{k}{2},0,0,0\right],
\left[\frac{k}{2}-2,1,0,0\right],
\left[\frac{k}{2}-3,0,2,0\right],
\left[\frac{k}{2}-2,0,1,1\right],
\left[\frac{k}{2}-1,0,1,0\right],
\\
&
\left[\frac{k}{2}-3,1,0,2\right],
\left[\frac{k}{2}-4,1,1,2\right],
\left[\frac{k}{2}-1,0,0,2\right],
\left[\frac{k}{2}-2,1,0,1\right],
\left[\frac{k}{2}-2,0,1,2\right],
\\
&
\left[\frac{k}{2}-3,2,0,0\right],
\left[\frac{k}{2}-3,1,0,3\right],
\left[\frac{k}{2}-1,1,0,0\right],
\left[\frac{k}{2}-2,1,0,2\right],
\left[\frac{k}{2}-3,2,0,2\right],
\\
&
\left[\frac{k}{2}-1,0,1,1\right],
\left[\frac{k}{2}-2,1,1,1\right],
\left[\frac{k}{2},0,0,1\right],
\left[\frac{k}{2}-1,0,2,0\right],
\left[\frac{k}{2}-1,1,0,1\right],
\\
&
\left[\frac{k}{2}-1,0,2,1\right],
\left[\frac{k}{2},0,1,0\right],
\left[\frac{k}{2},1,0,0\right],
\left[\frac{k}{2}+1,0,0,0\right]
\bigg\} 
\end{split}
\end{equation}
According to Lemma~\ref{lemma-orthsing}, the fusion kernel is controlled by this collection of weights together with generators for the modulo $2$ singular representation module $R[S]/(2)$.  Lemma~\ref{lemma-bs} identifies such generators in terms of the projection of the odd level weights in the basis $AB^{\text{ev}}_e$ to the level $k+1$ singular plane.  The image of that projection is the collection of six singular weights
\begin{equation} \nn
\begin{split}
B^{\text{ev}}_s = \bigg\{
&
\left[\frac{k}{2}-1,1,0,0\right],
\left[\frac{k}{2}-2,1,0,2\right],
\left[\frac{k}{2}-1,0,1,1\right],
\\
&
\left[\frac{k}{2},0,0,1\right],
\left[\frac{k}{2}-3,1,1,2\right],
\left[\frac{k}{2}-2,0,2,1\right]
\bigg\}
\end{split}
\end{equation}
These weights include the four singular weights of $AB^{\text{ev}}_e$, namely the 13-th, 14-th, 16-th, and 18-th weights in the above order, and only the last two weights of $B^{\text{ev}}_s$, that is $[\frac{k}{2} \!-\! 3,1,1,2]$ and $[\frac{k}{2} \!-\! 2,0,2,1]$, are not already weights of $AB^{\text{ev}}_e$.  Corollary~\ref{cor-orthog} assembles the fusion kernel in terms of this data: the given six singular weights $\lam \in B^{\text{ev}}_s$ give the elements $[\lam]_e$ in the kernel, while the nonsingular weights $\lam$ of $AB^{\text{ev}}_e$ give the elements $[\lam]_e + [w_{(k+1) \cdot \alpha_0} \lam]_e$ in the kernel.  Note that the 20-th, 22-nd, 23-rd, and 24-th weights of $AB^{\text{ev}}_e$, in the above order, produce the same element of the kernel as respectively the 9-th, 5-th, 2-nd, and 1-st weights.  Lemma~\ref{lemma-onevert} gives the fusion ideal listed in Theorem~\ref{thm-fusion}; there the representations associated to nonsingular weights of $AB^{\text{ev}}_e$ are listed first, followed by the representations associated to the singular weights $B^{\text{ev}}_s$.

The analysis for odd level $k$ is similar.  The weights $AB^{\text{odd}}_e = r(B_e) + (k+1) \frac{\lam_1}{2}$ form a Steinberg basis for $Z_e$.  (Note that this collection cannot be a level $k$ affine Steinberg basis for the pair $(Z_e, Z_v)$.)  Explicitly these weights are as follows:
\begin{equation} \nn
\begin{split}
AB^{\text{odd}}_e = \bigg\{
&
\left[\frac{k+1}{2},0,0,0\right],
\left[\frac{k-3}{2},1,0,0\right],
\left[\frac{k-5}{2},0,2,0\right],
\left[\frac{k-3}{2},0,1,1\right],
\left[\frac{k-1}{2},0,1,0\right],
\\
&
\left[\frac{k-5}{2},1,0,2\right],
\left[\frac{k-7}{2},1,1,2\right],
\left[\frac{k-1}{2},0,0,2\right],
\left[\frac{k-3}{2},1,0,1\right],
\left[\frac{k-3}{2},0,1,2\right],
\\
&
\left[\frac{k-5}{2},2,0,0\right],
\left[\frac{k-5}{2},1,0,3\right],
\left[\frac{k-1}{2},1,0,0\right],
\left[\frac{k-3}{2},1,0,2\right],
\left[\frac{k-5}{2},2,0,2\right],
\\
&
\left[\frac{k-1}{2},0,1,1\right],
\left[\frac{k-3}{2},1,1,1\right],
\left[\frac{k+1}{2},0,0,1\right],
\left[\frac{k-1}{2},0,2,0\right],
\left[\frac{k-1}{2},1,0,1\right],
\\
&
\left[\frac{k-1}{2},0,2,1\right],
\left[\frac{k+1}{2},0,1,0\right],
\left[\frac{k+1}{2},1,0,0\right],
\left[\frac{k+3}{2},0,0,0\right]
\bigg\}
\end{split}
\end{equation}
The projection of the even level weights in $AB^{\text{odd}}_e$ to the level $k+1$ singular plane is the collection
\begin{equation} \nn
\begin{split}
B^{\text{odd}}_s = \bigg\{
&
\left[\frac{k+1}{2},0,0,0\right],
\left[\frac{k-1}{2},0,1,0\right],
\left[\frac{k-1}{2},0,0,2\right],
\left[\frac{k-3}{2},1,0,1\right],
\left[\frac{k-3}{2},0,1,2\right],
\\
&
\left[\frac{k-5}{2},2,0,0\right],
\left[\frac{k-5}{2},1,0,3\right],
\left[\frac{k-3}{2},0,2,0\right],
\left[\frac{k-7}{2},2,0,2\right],
\left[\frac{k-5}{2},1,1,1\right]
\bigg\}
\end{split}
\end{equation}
These weights include the seven singular weights of $AB^{\text{odd}}_e$, namely the 1-st, 5-th, 8-th, 9-th, 10-th, 11-th, and 12-th weights in the above order, and only the last three weights of $B^{\text{odd}}_s$, that is $[\frac{k-3}{2},0,2,0]$, $[\frac{k-7}{2},2,0,2]$, and $[\frac{k-5}{2},1,1,1]$, are not already in $AB^{\text{odd}}_e$.  The 13-th, 14-th, 16-th, and 19-th weights of $AB^{\text{odd}}_e$ produce the same element of the fusion kernel as respectively the 2-nd, 6-th, 4-th, and 3-rd weights.  The resulting fusion ideal is given in Theorem~\ref{thm-fusion}; again there the representations associated to nonsingular weights of $AB^{\text{odd}}_e$ are listed first, followed by the representations associated to the singular weights $B^{\text{odd}}_s$.

\subsubsection*{$E_8$} 

Let $v$ be the vertex of the alcove corresponding to the simple root $\alpha_8$; this root is the terminal root of the long leg of the nonaffine Dynkin diagram.  Choose $\bar{D} \bs \alpha_8$ as the positive simple roots of the centralizer $Z_e= \langle E_7 \rangle$.  The $Z_e$-positive Weyl group is drawn in Figure~\ref{fig-e8}.  As in the corresponding figure for $F_4$, the rightmost edges are labelled by numbers $i$ and the remaining edges can be numbered as follows: if translation of an edge in a direction perpendicular to that edge produces another edge of the figure, those edges share a label; the labels of the central hexagons alternate $i$, $j$, $i$, $j$, $i$, $j$.  Again, each node $p$ of the figure represents an element of the $Z_e$-positive Weyl group $W^{{Z_e}_+}$, namely the element $w_p := w_{i_1} w_{i_2} \cdots w_{i_m}$ where $i_1, i_2, \ldots, i_m$ is the sequence of labels along a path from that node $p$ to the top node of the figure.  The $Z_e$-positive weight $\lam_{w_p}$ corresponding to the $Z_e$-positive Weyl element $w_p$ is $\lam_{w_p} = \sum_{i \searrow p} \lam_i$; this sum is over the labels on the edges directly descending to the node $p$.

\pagestyle{empty}
\thispagestyle{plain}
\begin{figure}[p]
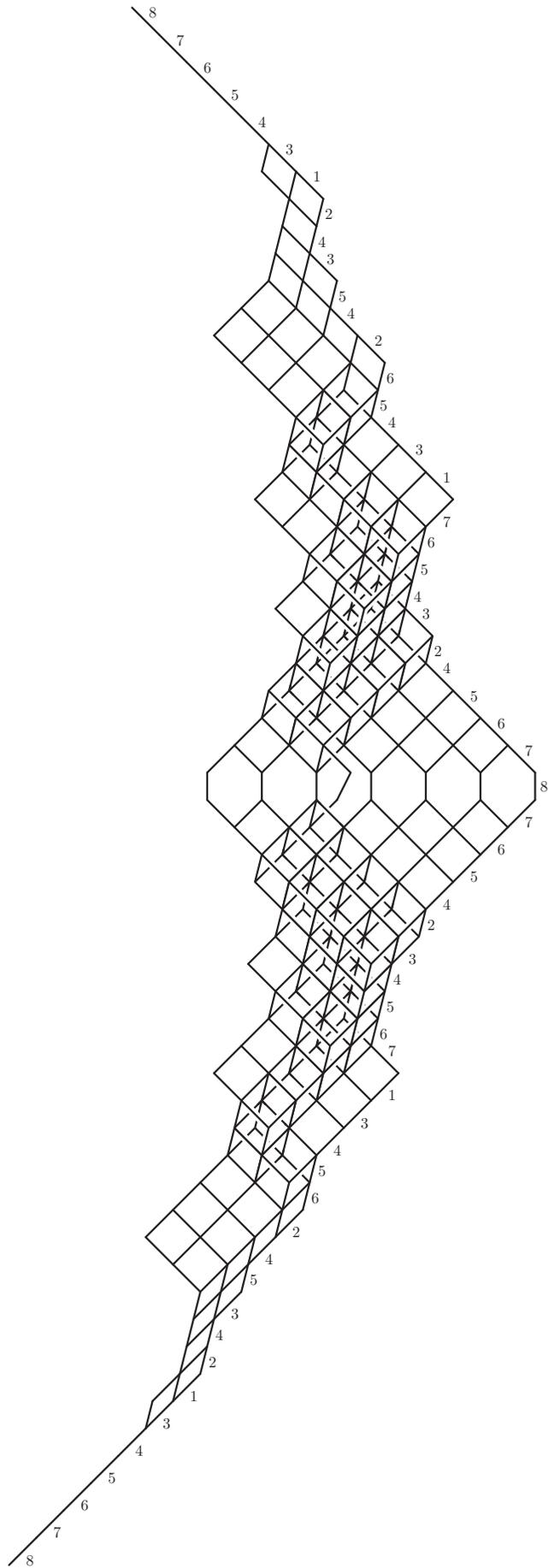

\begin{center}
\vspace*{-55pt} 
\ingm{e8fig}
\caption{The $\langle E_7 \rangle$-positive Weyl group of $E_8$} \label{fig-e8}
\end{center}
\end{figure}

We give an example of the above interpretation of the figure.  The eight edges in the vertical middle of the diagram are labelled, from left to right, $1, 3, 4, 2, 5, 6, 7, 8$.  The node $p$ at the top of that edge labelled $4$ represents the Weyl element $w_p = 3 1 2 5 4 3 6 5 4 2 7 6 5 4 3 1 8 7 6 5 4 3 2 4 5 6 7 8$; here we followed the lefthand edge of the diagram.  The positive weight corresponding to this node is $\lam_{w_p} = \lam_3 + \lam_2 + \lam_5$.

As usual, in terms of the $Z_e$-positive Weyl elements $W^{{Z_e}_+}$ and weights $\Lambda^{{Z_e}_+}$, the Steinberg basis for $R[Z_e]$ as an $R[G]$-module is $B_e = \{w^{-1} \lam_w \st w \in W^{{Z_e}_+}\}$.  Choose $(\bar{D} \cup -\alpha_0) \bs \alpha_8$ as the positive simple roots of the centralizer $Z_v=E_7 \times_{\ZZ/2} SU(2)$.  The $Z_v$-positive Weyl group $W^{{Z_v}_+}$ is the collection of Weyl elements associated to the nodes in the top half of Figure~\ref{fig-e8}, and the vertex Steinberg basis is $B_v = \{w^{-1} \lam_w \st w \in W^{{Z_v}_+}\}$.

\begin{remark}
Figure~\ref{fig-e8} provides a diagram of the $\langle E_7 \rangle$-positive Weyl group of $E_8$.  The reader may have noticed not only that the eight central edges of the diagram are arranged in the pattern of the Dynkin diagram of $E_8$, but that the entire figure is precisely the Hasse diagram of the roots of $E_8$.  This convergence is rather specialized---it does not hold, for instance, for any of the other centralizers we have considered in the preceding sections.  That said, there are entertaining combinatorial relationships among the $E$-exceptional positive Weyl groups.  For instance, both the $\langle E_6 \rangle$-positive Weyl group of $E_7$ and the $\langle D_5 \rangle$-positive Weyl group of $E_6$ embed into the $\langle E_7 \rangle$-positive Weyl group of $E_8$.  There are also more subtle resulting correspondences between the Steinberg bases for $(\langle D_5 \rangle \subset E_6)$ and $(\langle E_6 \rangle \subset E_7)$ and the Steinberg basis for $(\langle E_7 \rangle \subset E_8)$.  We can imagine that these phenomena are related to the ideas explored by Purbhoo~\cite{purbhoo-comp} concerning compression of exceptional root systems.
\end{remark}

Equipped with the Steinberg basis $B_e$ for $R[\langle E_7 \rangle]$ as an $R[E_8]$-module, we can complete the computation of the fusion ideal of $E_8$.  Suppose the level $k$ is even.  Let $r$ denote the Weyl reflection $w_0$.  The pair
\[
(AB^{\text{ev}}_e, AB^{\text{ev}}_v) = \Big(r(B_e) + k \frac{\lam_8}{2}, r(B_v) + k \frac{\lam_8}{2}\Big)
\]
is a level $k$ affine Steinberg basis for $(Z_e, Z_v)$.  The complement $AB^{\text{ev}}_e \bs AB^{\text{ev}}_v$ is the collection of weights $r(w_p^{-1} \lam_{w_p}) + k \frac{\lam_8}{2}$ for $p$ a node in the bottom half of Figure~\ref{fig-e8}.  The top 8 weights of that bottom half (in the vertical partial order of the figure) are level $k+1$, thus singular for induction to the representation module $R_k[Z_v]$.  The remaining 112 weights of $AB^{\text{ev}}_e \bs AB^{\text{ev}}_v$ all induce, by the reflection $w_{(k+1) \cdot \alpha_0}$, into the collection $AB^{\text{ev}}_v$.  This affine Steinberg basis is therefore induction closed in the sense of Section~\ref{sec-indclosed}.  Proposition~\ref{prop-indclosed} identifies the fusion kernel as 
\[
\big\{[\lam]_e \st \lam \in S_v\big\} \cup \big\{[\lam]_e + [w_{(k+1) \cdot \alpha_0} \lam]_e \st \lam \in AB^{\text{ev}}_e \bs (AB^{\text{ev}}_v \cup S_v)\big\}
\] 
where $S_v$ denotes the singular weights of $AB^{\text{ev}}_e$.  Lemma~\ref{lemma-onevert} in turn describes the fusion ideal, which is written out in Table~1.

\pagestyle{plain}
Finally suppose the level $k$ is odd.  The collection of 240 weights
\[AB^{\text{odd}}_e = r(B_e) + (k+1) \frac{\lam_8}{2}\]
is a Steinberg basis for $Z_e$.  The projection of the even level weights of $AB^{\text{odd}}_e$ to the level $k+1$ singular plane is a collection of 54 weights 
\[
B_s := \Big\{ \lam + \frac{k+1-\Lv(\lam)}{2} \lam_8 \st \lam \in AB^{\text{odd}}_e, \Lv(\lam) = k+1 \;(\mathrm{mod}\; 2)\Big\}
\]
By Lemma~\ref{lemma-bs} the irreducible representations associated to these weights generate the modulo 2 singular weight module $R[S]/(2)$.  That collection $B_s$ includes the 44 singular weights of $AB^{\text{odd}}_e$; the fusion kernel, as described in Lemma~\ref{lemma-orthsing}, is therefore generated by 
\[
\big\{[\lam]_e \st \lam \in B_s\big\} \cup \big\{[\lam]_e + [w_{(k+1) \cdot \alpha_0} \lam]_e \st \lam \in AB^{\text{odd}}_e \bs (AB^{\text{odd}}_e \cap S)\big\}
\]
Of the 120 weights of $AB^{\text{odd}}_e$ of level greater than $k+1$, 35 weights reflect by $w_{(k+1) \cdot \alpha_0}$ into the 76 weights of $AB^{\text{odd}}_e$ of level less than $k+1$.  The collection of representations $\{[\lam]_e + [w_{(k+1) \cdot \alpha_0} \lam]_e \st \lam \in AB^{\text{odd}}_e \bs (AB^{\text{odd}}_e \cap S)\}$ therefore has 161 elements; these representations together with the representations $\{[\lam]_e \st \lam \in B_s\}$ generate the fusion kernel.  The resulting fusion ideal is written out in Table~2-1 and its continuation Table~2-2.

\begin{table}
\caption{Generators of the fusion ideal of $E_8$ at even level $k = 2 \ell$.} \label{table-e8ev}
\Small
\[
\begin{array}{ll}
[0,1,0,0,0,0,0,\ell\!-\!1],
[0,1,0,0,0,1,0,\ell\!-\!3], &
[0,1,0,0,1,0,1,\ell\!-\!5],
[0,0,0,0,1,0,0,\ell\!-\!2], \\

[0,0,0,1,0,0,1,\ell\!-\!4],
[0,0,1,0,0,0,1,\ell\!-\!3], &
[1,0,0,0,0,0,1,\ell\!-\!2],
[0,0,0,0,0,0,1,\ell\!-\!1], \\

[0,2,0,0,0,1,0,\ell\!-\!4]+[0,2,0,0,0,1,0,\ell\!-\!5], &
[0,2,0,0,1,1,1,\ell\!-\!8]+[0,2,0,0,1,1,1,\ell\!-\!9], \\

[0,1,0,0,2,0,1,\ell\!-\!7]+[0,1,0,0,2,0,1,\ell\!-\!8], &
[0,1,0,1,1,0,2,\ell\!-\!9]+[0,1,0,1,1,0,2,\ell\!-\!10], \\

[0,0,1,1,0,0,2,\ell\!-\!7]+[0,0,1,1,0,0,2,\ell\!-\!8], &
[1,0,1,0,0,0,2,\ell\!-\!5]+[1,0,1,0,0,0,2,\ell\!-\!6], \\

[1,0,0,0,0,0,2,\ell\!-\!3]+[1,0,0,0,0,0,2,\ell\!-\!4], &
[0,2,0,0,1,0,1,\ell\!-\!6]+[0,2,0,0,1,0,1,\ell\!-\!7], \\

[0,1,0,0,1,1,0,\ell\!-\!5]+[0,1,0,0,1,1,0,\ell\!-\!6], &
[0,1,0,1,0,1,1,\ell\!-\!7]+[0,1,0,1,0,1,1,\ell\!-\!8], \\

[0,0,0,1,1,0,1,\ell\!-\!6]+[0,0,0,1,1,0,1,\ell\!-\!7], &
[0,1,1,0,1,0,2,\ell\!-\!8]+[0,1,1,0,1,0,2,\ell\!-\!9], \\

[1,0,0,1,0,0,2,\ell\!-\!6]+[1,0,0,1,0,0,2,\ell\!-\!7], &
[0,0,1,0,0,0,2,\ell\!-\!4]+[0,0,1,0,0,0,2,\ell\!-\!5], \\

[0,1,0,0,1,0,0,\ell\!-\!3]+[0,1,0,0,1,0,0,\ell\!-\!4], &
[0,1,0,1,0,0,1,\ell\!-\!5]+[0,1,0,1,0,0,1,\ell\!-\!6], \\

[0,1,0,1,1,1,1,\ell\!-\!9]+[0,1,0,1,1,1,1,\ell\!-\!11], &
[0,1,1,0,0,1,1,\ell\!-\!6]+[0,1,1,0,0,1,1,\ell\!-\!7], \\

[0,0,1,0,1,0,1,\ell\!-\!5]+[0,0,1,0,1,0,1,\ell\!-\!6], &
[1,1,0,0,1,0,2,\ell\!-\!7]+[1,1,0,0,1,0,2,\ell\!-\!8], \\

[0,0,0,1,0,0,2,\ell\!-\!5]+[0,0,0,1,0,0,2,\ell\!-\!6], &
[0,1,0,1,1,0,1,\ell\!-\!7]+[0,1,0,1,1,0,1,\ell\!-\!9], \\

[0,0,0,1,0,1,0,\ell\!-\!4]+[0,0,0,1,0,1,0,\ell\!-\!5], &
[0,1,1,0,0,0,1,\ell\!-\!4]+[0,1,1,0,0,0,1,\ell\!-\!5], \\

[0,1,1,0,1,1,1,\ell\!-\!8]+[0,1,1,0,1,1,1,\ell\!-\!10], &
[1,1,0,0,0,1,1,\ell\!-\!5]+[1,1,0,0,0,1,1,\ell\!-\!6], \\

[1,0,0,0,1,0,1,\ell\!-\!4]+[1,0,0,0,1,0,1,\ell\!-\!5], &
[0,1,0,0,1,0,2,\ell\!-\!6]+[0,1,0,0,1,0,2,\ell\!-\!7], \\

[0,1,0,1,0,1,0,\ell\!-\!5]+[0,1,0,1,0,1,0,\ell\!-\!7], &
[0,1,1,0,1,0,1,\ell\!-\!6]+[0,1,1,0,1,0,1,\ell\!-\!8], \\

[0,0,1,1,0,1,1,\ell\!-\!7]+[0,0,1,1,0,1,1,\ell\!-\!9], &
[1,1,0,0,0,0,1,\ell\!-\!3]+[1,1,0,0,0,0,1,\ell\!-\!4], \\

[1,1,0,0,1,1,1,\ell\!-\!7]+[1,1,0,0,1,1,1,\ell\!-\!9], &
[0,1,0,0,0,1,1,\ell\!-\!4]+[0,1,0,0,0,1,1,\ell\!-\!5], \\

[0,0,0,0,1,0,1,\ell\!-\!3]+[0,0,0,0,1,0,1,\ell\!-\!4], &
[0,0,0,1,0,0,0,\ell\!-\!2]+[0,0,0,1,0,0,0,\ell\!-\!3], \\

[0,1,1,1,0,1,1,\ell\!-\!8]+[0,1,1,1,0,1,1,\ell\!-\!11], &
[0,0,1,0,0,1,0,\ell\!-\!3]+[0,0,1,0,0,1,0,\ell\!-\!4], \\

[1,1,0,0,1,0,1,\ell\!-\!5]+[1,1,0,0,1,0,1,\ell\!-\!7], &
[1,0,0,1,0,1,1,\ell\!-\!6]+[1,0,0,1,0,1,1,\ell\!-\!8], \\

[0,1,0,0,0,0,1,\ell\!-\!2]+[0,1,0,0,0,0,1,\ell\!-\!3], &
[0,1,0,0,1,1,1,\ell\!-\!6]+[0,1,0,0,1,1,1,\ell\!-\!8], \\

[0,0,1,1,0,0,1,\ell\!-\!5]+[0,0,1,1,0,0,1,\ell\!-\!7], &
[0,1,1,0,0,1,0,\ell\!-\!4]+[0,1,1,0,0,1,0,\ell\!-\!6], \\

[1,1,0,1,0,1,1,\ell\!-\!7]+[1,1,0,1,0,1,1,\ell\!-\!10], &
[1,0,1,0,0,1,1,\ell\!-\!5]+[1,0,1,0,0,1,1,\ell\!-\!7], \\

[0,1,0,0,1,0,1,\ell\!-\!4]+[0,1,0,0,1,0,1,\ell\!-\!6], &
[0,0,0,1,0,1,1,\ell\!-\!5]+[0,0,0,1,0,1,1,\ell\!-\!7], \\

[0,0,1,1,0,1,0,\ell\!-\!5]+[0,0,1,1,0,1,0,\ell\!-\!8], &
[1,0,0,1,0,0,1,\ell\!-\!4]+[1,0,0,1,0,0,1,\ell\!-\!6], \\

[1,1,1,0,0,1,1,\ell\!-\!6]+[1,1,1,0,0,1,1,\ell\!-\!9], &
[1,0,0,0,0,1,0,\ell\!-\!2]+[1,0,0,0,0,1,0,\ell\!-\!3], \\

[0,1,0,1,0,1,1,\ell\!-\!6]+[0,1,0,1,0,1,1,\ell\!-\!9], &
[0,0,1,0,0,1,1,\ell\!-\!4]+[0,0,1,0,0,1,1,\ell\!-\!6], \\

[0,0,1,0,1,0,0,\ell\!-\!3]+[0,0,1,0,1,0,0,\ell\!-\!5], &
[1,0,1,1,0,1,1,\ell\!-\!7]+[1,0,1,1,0,1,1,\ell\!-\!11], \\

[1,1,0,0,0,1,0,\ell\!-\!3]+[1,1,0,0,0,1,0,\ell\!-\!5], &
[0,0,0,1,0,0,1,\ell\!-\!3]+[0,0,0,1,0,0,1,\ell\!-\!5], \\

[0,1,1,0,0,1,1,\ell\!-\!5]+[0,1,1,0,0,1,1,\ell\!-\!8], &
[1,0,0,0,0,1,1,\ell\!-\!3]+[1,0,0,0,0,1,1,\ell\!-\!5], \\

[0,0,1,0,0,0,0,\ell\!-\!1]+[0,0,1,0,0,0,0,\ell\!-\!2], &
[1,0,1,0,1,0,1,\ell\!-\!5]+[1,0,1,0,1,0,1,\ell\!-\!8], \\

[1,0,0,1,0,1,0,\ell\!-\!4]+[1,0,0,1,0,1,0,\ell\!-\!7], &
[0,0,1,1,0,1,1,\ell\!-\!6]+[0,0,1,1,0,1,1,\ell\!-\!10], \\

[1,1,0,0,0,1,1,\ell\!-\!4]+[1,1,0,0,0,1,1,\ell\!-\!7], &
[0,0,0,0,0,1,0,\ell\!-\!1]+[0,0,0,0,0,1,0,\ell\!-\!2], \\

[1,0,1,0,0,0,1,\ell\!-\!3]+[1,0,1,0,0,0,1,\ell\!-\!5], &
[1,0,1,0,1,1,0,\ell\!-\!5]+[1,0,1,0,1,1,0,\ell\!-\!9], \\

[0,0,1,0,1,0,1,\ell\!-\!4]+[0,0,1,0,1,0,1,\ell\!-\!7], &
[1,0,0,1,0,1,1,\ell\!-\!5]+[1,0,0,1,0,1,1,\ell\!-\!9], \\

[0,1,0,0,0,1,0,\ell\!-\!2]+[0,1,0,0,0,1,0,\ell\!-\!4], &
[1,0,1,0,0,1,0,\ell\!-\!3]+[1,0,1,0,0,1,0,\ell\!-\!6], \\

[1,0,0,0,1,0,0,\ell\!-\!2]+[1,0,0,0,1,0,0,\ell\!-\!4], &
[0,0,1,0,0,0,1,\ell\!-\!2]+[0,0,1,0,0,0,1,\ell\!-\!4], \\

[1,0,1,0,1,1,1,\ell\!-\!6]+[1,0,1,0,1,1,1,\ell\!-\!11], &
[0,0,0,1,0,1,0,\ell\!-\!3]+[0,0,0,1,0,1,0,\ell\!-\!6], \\

[1,0,1,0,1,0,0,\ell\!-\!3]+[1,0,1,0,1,0,0,\ell\!-\!7], &
[1,0,1,0,0,1,1,\ell\!-\!4]+[1,0,1,0,0,1,1,\ell\!-\!8], \\

[1,0,0,0,1,0,1,\ell\!-\!3]+[1,0,0,0,1,0,1,\ell\!-\!6], &
[0,0,1,0,1,1,0,\ell\!-\!4]+[0,0,1,0,1,1,0,\ell\!-\!8], \\

[1,0,0,1,0,0,0,\ell\!-\!2]+[1,0,0,1,0,0,0,\ell\!-\!5], &
[1,0,1,0,1,0,1,\ell\!-\!4]+[1,0,1,0,1,0,1,\ell\!-\!9], \\

[0,0,1,0,0,1,0,\ell\!-\!2]+[0,0,1,0,0,1,0,\ell\!-\!5], &
[1,0,0,0,1,1,0,\ell\!-\!3]+[1,0,0,0,1,1,0,\ell\!-\!7], \\

[1,1,0,0,0,0,0,\ell\!-\!1]+[1,1,0,0,0,0,0,\ell\!-\!3], &
[1,0,0,1,0,0,1,\ell\!-\!3]+[1,0,0,1,0,0,1,\ell\!-\!7], \\

[1,0,1,0,1,1,0,\ell\!-\!4]+[1,0,1,0,1,1,0,\ell\!-\!10], &
[0,0,0,0,1,0,0,\ell\!-\!1]+[0,0,0,0,1,0,0,\ell\!-\!3], \\

[1,0,0,0,0,0,0,\ell]+[1,0,0,0,0,0,0,\ell\!-\!1], &
[1,1,0,0,0,0,1,\ell\!-\!2]+[1,1,0,0,0,0,1,\ell\!-\!5], \\

[1,0,0,1,0,1,0,\ell\!-\!3]+[1,0,0,1,0,1,0,\ell\!-\!8], &
[0,0,1,0,1,0,0,\ell\!-\!2]+[0,0,1,0,1,0,0,\ell\!-\!6], \\

[1,0,0,0,0,0,1,\ell\!-\!1]+[1,0,0,0,0,0,1,\ell\!-\!3], &
[1,1,0,0,0,1,0,\ell\!-\!2]+[1,1,0,0,0,1,0,\ell\!-\!6], \\

[1,0,0,1,1,0,0,\ell\!-\!3]+[1,0,0,1,1,0,0,\ell\!-\!9], &
[1,0,0,0,0,1,0,\ell\!-\!1]+[1,0,0,0,0,1,0,\ell\!-\!4], \\

[1,1,0,0,1,0,0,\ell\!-\!2]+[1,1,0,0,1,0,0,\ell\!-\!7], &
[0,0,0,1,0,0,0,\ell\!-\!1]+[0,0,0,1,0,0,0,\ell\!-\!4], \\

[1,0,0,0,1,0,0,\ell\!-\!1]+[1,0,0,0,1,0,0,\ell\!-\!5], &
[1,1,0,1,0,0,0,\ell\!-\!2]+[1,1,0,1,0,0,0,\ell\!-\!8], \\

[1,0,0,1,0,0,0,\ell\!-\!1]+[1,0,0,1,0,0,0,\ell\!-\!6], &
[0,1,1,0,0,0,0,\ell\!-\!1]+[0,1,1,0,0,0,0,\ell\!-\!5], \\

[1,1,1,0,0,0,0,\ell\!-\!1]+[1,1,1,0,0,0,0,\ell\!-\!7], &
[0,1,0,0,0,0,0,\ell]+[0,1,0,0,0,0,0,\ell\!-\!2], \\

[0,0,1,0,0,0,0,\ell]+[0,0,1,0,0,0,0,\ell\!-\!3], &
[1,1,0,0,0,0,0,\ell]+[1,1,0,0,0,0,0,\ell\!-\!4], \\

[0,1,1,0,0,0,0,\ell]+[0,1,1,0,0,0,0,\ell\!-\!6], &
[0,0,0,1,0,0,0,\ell]+[0,0,0,1,0,0,0,\ell\!-\!5], \\

[0,0,0,0,1,0,0,\ell]+[0,0,0,0,1,0,0,\ell\!-\!4], &
[0,0,0,0,0,1,0,\ell]+[0,0,0,0,0,1,0,\ell\!-\!3], \\

[0,0,0,0,0,0,1,\ell]+[0,0,0,0,0,0,1,\ell\!-\!2], &
[0,0,0,0,0,0,0,\ell\!+\!1]+[0,0,0,0,0,0,0,\ell]
\end{array}
\]
\normalsize
\end{table}

\begin{table}
\renewcommand{\thetable}{2-1}
\caption{Generators of the fusion ideal of $E_8$ at odd level $k$, with $\langle n \rangle := \frac{n}{2}$.} \label{table-e8odd1}
\small
\[
\begin{array}{ll}
[0,0,0,0,0,0,1,\langle k\!-\!3 \rangle]+[0,0,0,0,0,0,1,\langle k\!-\!1 \rangle], &
[0,0,0,0,0,1,0,\langle k\!-\!5 \rangle]+[0,0,0,0,0,1,0,\langle k\!-\!1 \rangle], \\

[0,0,0,0,1,0,0,\langle k\!-\!7 \rangle]+[0,0,0,0,1,0,0,\langle k\!-\!1 \rangle], &
[0,0,0,1,0,0,0,\langle k\!-\!9 \rangle]+[0,0,0,1,0,0,0,\langle k\!-\!1 \rangle], \\

[0,1,1,0,0,0,0,\langle k\!-\!11 \rangle]+[0,1,1,0,0,0,0,\langle k\!-\!1 \rangle], &
[0,0,1,0,0,0,0,\langle k\!-\!5 \rangle]+[0,0,1,0,0,0,0,\langle k\!-\!1 \rangle], \\

[1,1,0,0,0,0,0,\langle k\!-\!7 \rangle]+[1,1,0,0,0,0,0,\langle k\!-\!1 \rangle], &
[1,1,1,0,0,0,0,\langle k\!-\!13 \rangle]+[1,1,1,0,0,0,0,\langle k\!-\!3 \rangle], \\

[0,1,0,0,0,0,0,\langle k\!-\!3 \rangle]+[0,1,0,0,0,0,0,\langle k\!-\!1 \rangle], &
[1,0,0,1,0,0,0,\langle k\!-\!11 \rangle]+[1,0,0,1,0,0,0,\langle k\!-\!3 \rangle], \\

[0,1,1,0,0,0,0,\langle k\!-\!9 \rangle]+[0,1,1,0,0,0,0,\langle k\!-\!3 \rangle], &
[1,0,0,0,1,0,0,\langle k\!-\!9 \rangle]+[1,0,0,0,1,0,0,\langle k\!-\!3 \rangle], \\

[1,1,0,1,0,0,0,\langle k\!-\!15 \rangle]+[1,1,0,1,0,0,0,\langle k\!-\!5 \rangle], &
[1,0,0,0,0,1,0,\langle k\!-\!7 \rangle]+[1,0,0,0,0,1,0,\langle k\!-\!3 \rangle], \\

[1,1,0,0,1,0,0,\langle k\!-\!13 \rangle]+[1,1,0,0,1,0,0,\langle k\!-\!5 \rangle], &
[0,0,0,1,0,0,0,\langle k\!-\!7 \rangle]+[0,0,0,1,0,0,0,\langle k\!-\!3 \rangle], \\

[1,0,0,0,0,0,1,\langle k\!-\!5 \rangle]+[1,0,0,0,0,0,1,\langle k\!-\!3 \rangle], &
[1,1,0,0,0,1,0,\langle k\!-\!11 \rangle]+[1,1,0,0,0,1,0,\langle k\!-\!5 \rangle], \\

[1,0,0,1,1,0,0,\langle k\!-\!17 \rangle]+[1,0,0,1,1,0,0,\langle k\!-\!7 \rangle], &
[1,1,0,0,0,0,1,\langle k\!-\!9 \rangle]+[1,1,0,0,0,0,1,\langle k\!-\!5 \rangle], \\

[1,0,0,1,0,1,0,\langle k\!-\!15 \rangle]+[1,0,0,1,0,1,0,\langle k\!-\!7 \rangle], &
[0,0,1,0,1,0,0,\langle k\!-\!11 \rangle]+[0,0,1,0,1,0,0,\langle k\!-\!5 \rangle], \\

[1,1,0,0,0,0,0,\langle k\!-\!5 \rangle]+[1,1,0,0,0,0,0,\langle k\!-\!3 \rangle], &
[1,0,0,1,0,0,1,\langle k\!-\!13 \rangle]+[1,0,0,1,0,0,1,\langle k\!-\!7 \rangle], \\

[1,0,1,0,1,1,0,\langle k\!-\!19 \rangle]+[1,0,1,0,1,1,0,\langle k\!-\!9 \rangle], &
[0,0,0,0,1,0,0,\langle k\!-\!5 \rangle]+[0,0,0,0,1,0,0,\langle k\!-\!3 \rangle], \\

[1,0,0,1,0,0,0,\langle k\!-\!9 \rangle]+[1,0,0,1,0,0,0,\langle k\!-\!5 \rangle], &
[1,0,1,0,1,0,1,\langle k\!-\!17 \rangle]+[1,0,1,0,1,0,1,\langle k\!-\!9 \rangle], \\

[0,0,1,0,0,1,0,\langle k\!-\!9 \rangle]+[0,0,1,0,0,1,0,\langle k\!-\!5 \rangle], &
[1,0,0,0,1,1,0,\langle k\!-\!13 \rangle]+[1,0,0,0,1,1,0,\langle k\!-\!7 \rangle], \\

[1,0,1,0,1,0,0,\langle k\!-\!13 \rangle]+[1,0,1,0,1,0,0,\langle k\!-\!7 \rangle], &
[1,0,1,0,0,1,1,\langle k\!-\!15 \rangle]+[1,0,1,0,0,1,1,\langle k\!-\!9 \rangle], \\

[1,0,0,0,1,0,1,\langle k\!-\!11 \rangle]+[1,0,0,0,1,0,1,\langle k\!-\!7 \rangle], &
[0,0,1,0,1,1,0,\langle k\!-\!15 \rangle]+[0,0,1,0,1,1,0,\langle k\!-\!9 \rangle], \\

[1,0,1,0,0,1,0,\langle k\!-\!11 \rangle]+[1,0,1,0,0,1,0,\langle k\!-\!7 \rangle], &
[0,0,1,0,0,0,1,\langle k\!-\!7 \rangle]+[0,0,1,0,0,0,1,\langle k\!-\!5 \rangle], \\

[1,0,0,0,1,0,0,\langle k\!-\!7 \rangle]+[1,0,0,0,1,0,0,\langle k\!-\!5 \rangle], &
[1,0,1,0,1,1,1,\langle k\!-\!21 \rangle]+[1,0,1,0,1,1,1,\langle k\!-\!13 \rangle], \\

[0,0,0,1,0,1,0,\langle k\!-\!11 \rangle]+[0,0,0,1,0,1,0,\langle k\!-\!7 \rangle], &
[1,0,1,0,0,0,1,\langle k\!-\!9 \rangle]+[1,0,1,0,0,0,1,\langle k\!-\!7 \rangle], \\

[1,0,1,0,1,1,0,\langle k\!-\!17 \rangle]+[1,0,1,0,1,1,0,\langle k\!-\!11 \rangle], &
[0,0,1,0,1,0,1,\langle k\!-\!13 \rangle]+[0,0,1,0,1,0,1,\langle k\!-\!9 \rangle], \\

[1,0,0,1,0,1,1,\langle k\!-\!17 \rangle]+[1,0,0,1,0,1,1,\langle k\!-\!11 \rangle], &
[0,1,0,0,0,1,0,\langle k\!-\!7 \rangle]+[0,1,0,0,0,1,0,\langle k\!-\!5 \rangle], \\

[1,0,1,0,1,0,1,\langle k\!-\!15 \rangle]+[1,0,1,0,1,0,1,\langle k\!-\!11 \rangle], &
[1,0,0,1,0,1,0,\langle k\!-\!13 \rangle]+[1,0,0,1,0,1,0,\langle k\!-\!9 \rangle], \\

[0,0,1,1,0,1,1,\langle k\!-\!19 \rangle]+[0,0,1,1,0,1,1,\langle k\!-\!13 \rangle], &
[1,1,0,0,0,1,1,\langle k\!-\!13 \rangle]+[1,1,0,0,0,1,1,\langle k\!-\!9 \rangle], \\

[0,0,1,0,1,0,0,\langle k\!-\!9 \rangle]+[0,0,1,0,1,0,0,\langle k\!-\!7 \rangle], &
[1,0,1,1,0,1,1,\langle k\!-\!21 \rangle]+[1,0,1,1,0,1,1,\langle k\!-\!15 \rangle], \\

[0,0,0,1,0,0,1,\langle k\!-\!9 \rangle]+[0,0,0,1,0,0,1,\langle k\!-\!7 \rangle], &
[1,1,0,0,0,1,0,\langle k\!-\!9 \rangle]+[1,1,0,0,0,1,0,\langle k\!-\!7 \rangle], \\

[0,1,1,0,0,1,1,\langle k\!-\!15 \rangle]+[0,1,1,0,0,1,1,\langle k\!-\!11 \rangle], &
[1,0,0,0,0,1,1,\langle k\!-\!9 \rangle]+[1,0,0,0,0,1,1,\langle k\!-\!7 \rangle], \\

[0,0,1,1,0,1,0,\langle k\!-\!15 \rangle]+[0,0,1,1,0,1,0,\langle k\!-\!11 \rangle], &
[1,0,0,1,0,0,1,\langle k\!-\!11 \rangle]+[1,0,0,1,0,0,1,\langle k\!-\!9 \rangle], \\

[1,1,1,0,0,1,1,\langle k\!-\!17 \rangle]+[1,1,1,0,0,1,1,\langle k\!-\!13 \rangle], &
[0,1,0,1,0,1,1,\langle k\!-\!17 \rangle]+[0,1,0,1,0,1,1,\langle k\!-\!13 \rangle], \\

[0,0,1,0,0,1,1,\langle k\!-\!11 \rangle]+[0,0,1,0,0,1,1,\langle k\!-\!9 \rangle], &
[0,0,1,1,0,0,1,\langle k\!-\!13 \rangle]+[0,0,1,1,0,0,1,\langle k\!-\!11 \rangle], \\

[0,1,1,0,0,1,0,\langle k\!-\!11 \rangle]+[0,1,1,0,0,1,0,\langle k\!-\!9 \rangle], &
[1,1,0,1,0,1,1,\langle k\!-\!19 \rangle]+[1,1,0,1,0,1,1,\langle k\!-\!15 \rangle], \\

[0,1,0,0,1,0,1,\langle k\!-\!11 \rangle]+[0,1,0,0,1,0,1,\langle k\!-\!9 \rangle], &
[1,0,1,0,0,1,1,\langle k\!-\!13 \rangle]+[1,0,1,0,0,1,1,\langle k\!-\!11 \rangle], \\

[0,0,0,1,0,1,1,\langle k\!-\!13 \rangle]+[0,0,0,1,0,1,1,\langle k\!-\!11 \rangle], &
[0,1,1,1,0,1,1,\langle k\!-\!21 \rangle]+[0,1,1,1,0,1,1,\langle k\!-\!17 \rangle], \\

[1,1,0,0,1,0,1,\langle k\!-\!13 \rangle]+[1,1,0,0,1,0,1,\langle k\!-\!11 \rangle], &
[1,0,0,1,0,1,1,\langle k\!-\!15 \rangle]+[1,0,0,1,0,1,1,\langle k\!-\!13 \rangle], \\

[0,1,0,0,1,1,1,\langle k\!-\!15 \rangle]+[0,1,0,0,1,1,1,\langle k\!-\!13 \rangle], &
[0,1,0,1,0,1,0,\langle k\!-\!13 \rangle]+[0,1,0,1,0,1,0,\langle k\!-\!11 \rangle], \\

[0,1,1,0,1,0,1,\langle k\!-\!15 \rangle]+[0,1,1,0,1,0,1,\langle k\!-\!13 \rangle], &
[0,0,1,1,0,1,1,\langle k\!-\!17 \rangle]+[0,0,1,1,0,1,1,\langle k\!-\!15 \rangle], \\

[1,1,0,0,1,1,1,\langle k\!-\!17 \rangle]+[1,1,0,0,1,1,1,\langle k\!-\!15 \rangle], &
[0,1,0,1,1,0,1,\langle k\!-\!17 \rangle]+[0,1,0,1,1,0,1,\langle k\!-\!15 \rangle], \\

[0,1,1,0,1,1,1,\langle k\!-\!19 \rangle]+[0,1,1,0,1,1,1,\langle k\!-\!17 \rangle], &
[0,1,0,1,1,1,1,\langle k\!-\!21 \rangle]+[0,1,0,1,1,1,1,\langle k\!-\!19 \rangle], \\

[0,2,0,0,0,1,0,\langle k\!-\!7 \rangle]+[0,2,0,0,0,1,0,\langle k\!-\!11 \rangle], &
[0,2,0,0,1,1,1,\langle k\!-\!15 \rangle]+[0,2,0,0,1,1,1,\langle k\!-\!19 \rangle], \\

[0,1,0,0,2,0,1,\langle k\!-\!13 \rangle]+[0,1,0,0,2,0,1,\langle k\!-\!17 \rangle], &
[0,1,0,1,1,0,2,\langle k\!-\!17 \rangle]+[0,1,0,1,1,0,2,\langle k\!-\!21 \rangle], \\

[0,0,1,1,0,0,2,\langle k\!-\!13 \rangle]+[0,0,1,1,0,0,2,\langle k\!-\!17 \rangle], &
[1,0,1,0,0,0,2,\langle k\!-\!9 \rangle]+[1,0,1,0,0,0,2,\langle k\!-\!13 \rangle], \\

[1,0,0,0,0,0,2,\langle k\!-\!5 \rangle]+[1,0,0,0,0,0,2,\langle k\!-\!9 \rangle], &
[0,2,0,0,1,0,1,\langle k\!-\!11 \rangle]+[0,2,0,0,1,0,1,\langle k\!-\!15 \rangle], \\

[0,1,0,0,1,1,0,\langle k\!-\!9 \rangle]+[0,1,0,0,1,1,0,\langle k\!-\!13 \rangle], &
[0,0,0,1,1,0,1,\langle k\!-\!11 \rangle]+[0,0,0,1,1,0,1,\langle k\!-\!15 \rangle], \\

[0,1,1,0,1,0,2,\langle k\!-\!15 \rangle]+[0,1,1,0,1,0,2,\langle k\!-\!19 \rangle], &
[1,0,0,1,0,0,2,\langle k\!-\!11 \rangle]+[1,0,0,1,0,0,2,\langle k\!-\!15 \rangle], \\

[0,0,1,0,0,0,2,\langle k\!-\!7 \rangle]+[0,0,1,0,0,0,2,\langle k\!-\!11 \rangle], &
[0,1,0,0,1,0,0,\langle k\!-\!5 \rangle]+[0,1,0,0,1,0,0,\langle k\!-\!9 \rangle], \\

[0,1,0,1,0,0,1,\langle k\!-\!9 \rangle]+[0,1,0,1,0,0,1,\langle k\!-\!13 \rangle], &
[0,1,0,1,1,1,1,\langle k\!-\!17 \rangle]+[0,1,0,1,1,1,1,\langle k\!-\!23 \rangle], \\

[1,1,0,0,1,0,2,\langle k\!-\!13 \rangle]+[1,1,0,0,1,0,2,\langle k\!-\!17 \rangle], &
[0,0,0,1,0,0,2,\langle k\!-\!9 \rangle]+[0,0,0,1,0,0,2,\langle k\!-\!13 \rangle], 
\end{array}
\]
\end{table}

\begin{table}
\renewcommand{\thetable}{2-2}
\caption{Generators of the fusion ideal of $E_8$ at odd level $k$, with $\langle n \rangle := \frac{n}{2}$.  (Cont.)} \label{table-e8odd2}
\small
\[
\begin{array}{ll}
[0,1,0,1,1,0,1,\langle k\!-\!13 \rangle]+[0,1,0,1,1,0,1,\langle k\!-\!19 \rangle], &
[0,1,1,0,0,0,1,\langle k\!-\!7 \rangle]+[0,1,1,0,0,0,1,\langle k\!-\!11 \rangle], \\

[0,1,1,0,1,1,1,\langle k\!-\!15 \rangle]+[0,1,1,0,1,1,1,\langle k\!-\!21 \rangle], &
[0,1,0,0,1,0,2,\langle k\!-\!11 \rangle]+[0,1,0,0,1,0,2,\langle k\!-\!15 \rangle], \\

[0,1,0,1,0,1,0,\langle k\!-\!9 \rangle]+[0,1,0,1,0,1,0,\langle k\!-\!15 \rangle], &
[0,1,1,0,1,0,1,\langle k\!-\!11 \rangle]+[0,1,1,0,1,0,1,\langle k\!-\!17 \rangle], \\

[1,1,0,0,1,1,1,\langle k\!-\!13 \rangle]+[1,1,0,0,1,1,1,\langle k\!-\!19 \rangle], &
[0,1,0,0,0,1,1,\langle k\!-\!7 \rangle]+[0,1,0,0,0,1,1,\langle k\!-\!11 \rangle], \\

[0,0,0,0,1,0,1,\langle k\!-\!5 \rangle]+[0,0,0,0,1,0,1,\langle k\!-\!9 \rangle], &
[0,1,1,1,0,1,1,\langle k\!-\!15 \rangle]+[0,1,1,1,0,1,1,\langle k\!-\!23 \rangle], \\

[1,1,0,0,1,0,1,\langle k\!-\!9 \rangle]+[1,1,0,0,1,0,1,\langle k\!-\!15 \rangle], &
[0,1,0,0,0,0,1,\langle k\!-\!3 \rangle]+[0,1,0,0,0,0,1,\langle k\!-\!7 \rangle], \\

[0,1,0,0,1,1,1,\langle k\!-\!11 \rangle]+[0,1,0,0,1,1,1,\langle k\!-\!17 \rangle], &
[0,0,1,1,0,0,1,\langle k\!-\!9 \rangle]+[0,0,1,1,0,0,1,\langle k\!-\!15 \rangle], \\

[0,1,1,0,0,1,0,\langle k\!-\!7 \rangle]+[0,1,1,0,0,1,0,\langle k\!-\!13 \rangle], &
[1,1,0,1,0,1,1,\langle k\!-\!13 \rangle]+[1,1,0,1,0,1,1,\langle k\!-\!21 \rangle], \\

[0,1,0,0,1,0,1,\langle k\!-\!7 \rangle]+[0,1,0,0,1,0,1,\langle k\!-\!13 \rangle], &
[0,0,0,1,0,1,1,\langle k\!-\!9 \rangle]+[0,0,0,1,0,1,1,\langle k\!-\!15 \rangle], \\

[0,0,1,1,0,1,0,\langle k\!-\!9 \rangle]+[0,0,1,1,0,1,0,\langle k\!-\!17 \rangle], &
[1,1,1,0,0,1,1,\langle k\!-\!11 \rangle]+[1,1,1,0,0,1,1,\langle k\!-\!19 \rangle], \\

[0,1,0,1,0,1,1,\langle k\!-\!11 \rangle]+[0,1,0,1,0,1,1,\langle k\!-\!19 \rangle], &
[0,0,1,0,0,1,1,\langle k\!-\!7 \rangle]+[0,0,1,0,0,1,1,\langle k\!-\!13 \rangle], \\

[1,0,1,1,0,1,1,\langle k\!-\!13 \rangle]+[1,0,1,1,0,1,1,\langle k\!-\!23 \rangle], &
[0,0,0,1,0,0,1,\langle k\!-\!5 \rangle]+[0,0,0,1,0,0,1,\langle k\!-\!11 \rangle], \\

[0,1,1,0,0,1,1,\langle k\!-\!9 \rangle]+[0,1,1,0,0,1,1,\langle k\!-\!17 \rangle], &
[1,0,0,0,0,1,1,\langle k\!-\!5 \rangle]+[1,0,0,0,0,1,1,\langle k\!-\!11 \rangle], \\

[0,0,1,1,0,1,1,\langle k\!-\!11 \rangle]+[0,0,1,1,0,1,1,\langle k\!-\!21 \rangle], &
[1,1,0,0,0,1,1,\langle k\!-\!7 \rangle]+[1,1,0,0,0,1,1,\langle k\!-\!15 \rangle], \\

[1,0,1,0,0,0,1,\langle k\!-\!5 \rangle]+[1,0,1,0,0,0,1,\langle k\!-\!11 \rangle], &
[0,0,1,0,1,0,1,\langle k\!-\!7 \rangle]+[0,0,1,0,1,0,1,\langle k\!-\!15 \rangle], \\

[1,0,0,1,0,1,1,\langle k\!-\!9 \rangle]+[1,0,0,1,0,1,1,\langle k\!-\!19 \rangle], &
[0,1,0,0,0,1,0,\langle k\!-\!3 \rangle]+[0,1,0,0,0,1,0,\langle k\!-\!9 \rangle], \\

[1,0,1,0,0,1,0,\langle k\!-\!5 \rangle]+[1,0,1,0,0,1,0,\langle k\!-\!13 \rangle], &
[0,0,1,0,0,0,1,\langle k\!-\!3 \rangle]+[0,0,1,0,0,0,1,\langle k\!-\!9 \rangle], \\

[1,0,1,0,1,1,1,\langle k\!-\!11 \rangle]+[1,0,1,0,1,1,1,\langle k\!-\!23 \rangle], &
[0,0,0,1,0,1,0,\langle k\!-\!5 \rangle]+[0,0,0,1,0,1,0,\langle k\!-\!13 \rangle], \\

[1,0,1,0,1,0,0,\langle k\!-\!5 \rangle]+[1,0,1,0,1,0,0,\langle k\!-\!15 \rangle], &
[1,0,1,0,0,1,1,\langle k\!-\!7 \rangle]+[1,0,1,0,0,1,1,\langle k\!-\!17 \rangle], \\

[1,0,0,0,1,0,1,\langle k\!-\!5 \rangle]+[1,0,0,0,1,0,1,\langle k\!-\!13 \rangle], &
[0,0,1,0,1,1,0,\langle k\!-\!7 \rangle]+[0,0,1,0,1,1,0,\langle k\!-\!17 \rangle], \\

[1,0,1,0,1,0,1,\langle k\!-\!7 \rangle]+[1,0,1,0,1,0,1,\langle k\!-\!19 \rangle], &
[0,0,1,0,0,1,0,\langle k\!-\!3 \rangle]+[0,0,1,0,0,1,0,\langle k\!-\!11 \rangle], \\

[1,0,0,0,1,1,0,\langle k\!-\!5 \rangle]+[1,0,0,0,1,1,0,\langle k\!-\!15 \rangle], &
[1,0,0,1,0,0,1,\langle k\!-\!5 \rangle]+[1,0,0,1,0,0,1,\langle k\!-\!15 \rangle], \\

[1,0,1,0,1,1,0,\langle k\!-\!7 \rangle]+[1,0,1,0,1,1,0,\langle k\!-\!21 \rangle], &
[1,0,0,0,0,0,0,\langle k\!+\!1 \rangle]+[1,0,0,0,0,0,0,\langle k\!-\!3 \rangle], \\

[1,1,0,0,0,0,1,\langle k\!-\!3 \rangle]+[1,1,0,0,0,0,1,\langle k\!-\!11 \rangle], &
[1,0,0,1,0,1,0,\langle k\!-\!5 \rangle]+[1,0,0,1,0,1,0,\langle k\!-\!17 \rangle], \\

[0,0,1,0,1,0,0,\langle k\!-\!3 \rangle]+[0,0,1,0,1,0,0,\langle k\!-\!13 \rangle], &
[1,0,0,0,0,0,1,\langle k\!-\!1 \rangle]+[1,0,0,0,0,0,1,\langle k\!-\!7 \rangle], \\

[1,1,0,0,0,1,0,\langle k\!-\!3 \rangle]+[1,1,0,0,0,1,0,\langle k\!-\!13 \rangle], &
[1,0,0,1,1,0,0,\langle k\!-\!5 \rangle]+[1,0,0,1,1,0,0,\langle k\!-\!19 \rangle], \\

[1,0,0,0,0,1,0,\langle k\!-\!1 \rangle]+[1,0,0,0,0,1,0,\langle k\!-\!9 \rangle], &
[1,1,0,0,1,0,0,\langle k\!-\!3 \rangle]+[1,1,0,0,1,0,0,\langle k\!-\!15 \rangle], \\

[1,0,0,0,1,0,0,\langle k\!-\!1 \rangle]+[1,0,0,0,1,0,0,\langle k\!-\!11 \rangle], &
[1,1,0,1,0,0,0,\langle k\!-\!3 \rangle]+[1,1,0,1,0,0,0,\langle k\!-\!17 \rangle], \\

[1,0,0,1,0,0,0,\langle k\!-\!1 \rangle]+[1,0,0,1,0,0,0,\langle k\!-\!13 \rangle], &
[1,1,1,0,0,0,0,\langle k\!-\!1 \rangle]+[1,1,1,0,0,0,0,\langle k\!-\!15 \rangle], \\

[0,1,0,0,0,0,0,\langle k\!+\!1 \rangle]+[0,1,0,0,0,0,0,\langle k\!-\!5 \rangle], &
[0,0,1,0,0,0,0,\langle k\!+\!1 \rangle]+[0,0,1,0,0,0,0,\langle k\!-\!7 \rangle], \\

[1,1,0,0,0,0,0,\langle k\!+\!1 \rangle]+[1,1,0,0,0,0,0,\langle k\!-\!9 \rangle], &
[0,1,1,0,0,0,0,\langle k\!+\!1 \rangle]+[0,1,1,0,0,0,0,\langle k\!-\!13 \rangle], \\

[0,0,0,1,0,0,0,\langle k\!+\!1 \rangle]+[0,0,0,1,0,0,0,\langle k\!-\!11 \rangle], &
[0,0,0,0,1,0,0,\langle k\!+\!1 \rangle]+[0,0,0,0,1,0,0,\langle k\!-\!9 \rangle], \\

[0,0,0,0,0,1,0,\langle k\!+\!1 \rangle]+[0,0,0,0,0,1,0,\langle k\!-\!7 \rangle], &
[0,0,0,0,0,0,1,\langle k\!+\!1 \rangle]+[0,0,0,0,0,0,1,\langle k\!-\!5 \rangle], \\

[0,0,0,0,0,0,0,\langle k\!+\!3 \rangle]+[0,0,0,0,0,0,0,\langle k\!-\!1 \rangle], &
[0,0,0,0,0,0,0,\langle k\!+\!1 \rangle],
[1,0,0,0,0,0,0,\langle k\!-\!1 \rangle], \\

[0,0,1,0,0,0,0,\langle k\!-\!3 \rangle],
[0,0,0,0,0,1,0,\langle k\!-\!3 \rangle], &
[1,0,0,0,0,1,0,\langle k\!-\!5 \rangle],
[0,0,0,1,0,0,0,\langle k\!-\!5 \rangle], \\

[0,0,1,0,0,1,0,\langle k\!-\!7 \rangle],
[0,1,0,0,0,0,1,\langle k\!-\!5 \rangle], &
[1,1,0,0,0,0,1,\langle k\!-\!7 \rangle],
[0,1,0,0,0,1,1,\langle k\!-\!9 \rangle], \\

[0,0,0,0,1,0,1,\langle k\!-\!7 \rangle],
[0,1,1,0,0,0,1,\langle k\!-\!9 \rangle], &
[0,0,0,1,0,1,0,\langle k\!-\!9 \rangle],
[1,1,0,0,0,1,1,\langle k\!-\!11 \rangle], \\

[1,0,0,0,1,0,1,\langle k\!-\!9 \rangle],
[0,1,0,0,1,0,2,\langle k\!-\!13 \rangle], &
[0,1,0,0,1,0,0,\langle k\!-\!7 \rangle],
[0,1,0,1,0,0,1,\langle k\!-\!11 \rangle], \\

[0,1,1,0,0,1,1,\langle k\!-\!13 \rangle],
[0,0,1,0,1,0,1,\langle k\!-\!11 \rangle], &
[1,1,0,0,1,0,2,\langle k\!-\!15 \rangle],
[0,0,0,1,0,0,2,\langle k\!-\!11 \rangle], \\

[0,2,0,0,1,0,1,\langle k\!-\!13 \rangle],
[0,1,0,0,1,1,0,\langle k\!-\!11 \rangle], &
[0,1,0,1,0,1,1,\langle k\!-\!15 \rangle],
[0,0,0,1,1,0,1,\langle k\!-\!13 \rangle], \\

[0,1,1,0,1,0,2,\langle k\!-\!17 \rangle],
[1,0,0,1,0,0,2,\langle k\!-\!13 \rangle], &
[0,0,1,0,0,0,2,\langle k\!-\!9 \rangle],
[0,2,0,0,0,1,0,\langle k\!-\!9 \rangle], \\

[0,2,0,0,1,1,1,\langle k\!-\!17 \rangle],
[0,1,0,0,2,0,1,\langle k\!-\!15 \rangle], &
[0,1,0,1,1,0,2,\langle k\!-\!19 \rangle],
[0,0,1,1,0,0,2,\langle k\!-\!15 \rangle], \\

[1,0,1,0,0,0,2,\langle k\!-\!11 \rangle],
[1,0,0,0,0,0,2,\langle k\!-\!7 \rangle], &
[0,2,0,0,0,0,0,\langle k\!-\!5 \rangle],
[0,2,0,0,0,2,0,\langle k\!-\!13 \rangle], \\

[0,2,0,0,2,0,2,\langle k\!-\!21 \rangle],
[0,0,0,0,2,0,0,\langle k\!-\!9 \rangle], &
[0,0,0,2,0,0,2,\langle k\!-\!17 \rangle],
[0,0,2,0,0,0,2,\langle k\!-\!13 \rangle], \\

[2,0,0,0,0,0,2,\langle k\!-\!9 \rangle],
[0,0,0,0,0,0,2,\langle k\!-\!5 \rangle], &
[1,0,0,1,0,0,0,\langle k\!-\!7 \rangle],
[1,1,0,0,1,0,0,\langle k\!-\!9 \rangle], \\

[1,0,0,1,0,1,0,\langle k\!-\!11 \rangle],
[1,0,1,0,1,0,1,\langle k\!-\!13 \rangle], &
[1,0,1,0,0,1,0,\langle k\!-\!9 \rangle],
[1,0,1,0,1,1,1,\langle k\!-\!17 \rangle], \\

[0,0,1,1,0,1,0,\langle k\!-\!13 \rangle],
[1,1,1,0,0,1,1,\langle k\!-\!15 \rangle], &
[1,1,0,1,0,1,1,\langle k\!-\!17 \rangle],
[0,1,1,1,0,1,1,\langle k\!-\!19 \rangle]
\end{array}
\]
\end{table}

\newpage

\bibliography{fusionii}
\bibliographystyle{plain}

\end{document}